\documentclass[12pt,reqno]{amsart}
\usepackage{amsmath,amssymb,amsthm,mathtools,wasysym,calc,verbatim,enumitem,tikz,pgfplots,url,hyperref,mathrsfs,cite,fullpage,bbm}

\usepackage{comment}

\newtheorem{theorem}{Theorem}[section]
\newtheorem{lemma}[theorem]{Lemma}
\newtheorem{corollary}[theorem]{Corollary}

\newtheorem{observation}[theorem]{Observation}

\newtheorem{conjecture}[theorem]{Conjecture}

\newtheorem{claim}[theorem]{Claim}
\newtheorem{fact}[theorem]{Fact}

\theoremstyle{definition}

\theoremstyle{remark}
\newtheorem{remark}[theorem]{Remark}

\newcommand\N{\mathbb{N}}
\newcommand\R{\mathbb{R}}
\newcommand\Z{\mathbb{Z}}

\newcommand\eps{\varepsilon}

\renewcommand{\leq}{\leqslant}
\renewcommand{\geq}{\geqslant}

\renewcommand{\to}{\rightarrow}
\renewcommand{\Re}{\re}

\def\eps{\varepsilon}
	\def\p{\partial}

	\def\E{\mathbb{E}}
	\def\C{\mathbb{C}}
	\def\R{\mathbb{R}}
	\def\bH{\mathbb{H}}	
	\def\Z{\mathbb{Z}}
	\def\N{\mathbb{N}}
	\def\PP{\mathbb{P}}
	\def\l{\lambda}
	\def\k{\kappa}
	
	\def\s{\sigma}
	\def\t{\theta}
	\def\a{\alpha}
	\def\g{\gamma}
	\def\z{\zeta}	
	\def\zbar{\bar{z}}

	\def\Var{\mathrm{Var}}
	\def\im{\mathrm{Im}}
	\def\re{\mathrm{Re}}

	\def\NN{\mathcal{N}}
	
	\def\EE{\mathbb{E}}

	\def\vp{\varphi}

\begin{document}

\title{Central limit theorems and the geometry of polynomials}
\author{Marcus Michelen}
\address{Department of Mathematics, Statistics, and Computer Science. University of Illinois at Chicago.  Chicago, IL 60607, USA.  }
\email{michelen.math@gmail.com}
\author{Julian Sahasrabudhe}
\address{Department of Pure Mathematics and Mathematical Statistics, Wilberforce Road, University of Cambridge. Cambridge, CB3 0WA, UK.}
\email{jdrs2@cam.ac.uk}

\begin{abstract}
Let $X \in \{0,\ldots,n \}$ be a random variable, with mean $\mu$ and standard deviation $\s$ and let 
\[f_X(z) = \sum_{k} \PP(X = k) z^k, 
\] be its probability generating function. Pemantle conjectured that if $\s$ is large and $f_X$ has no roots close to $1\in \C$ then 
$X$ must be approximately normal. We completely resolve this conjecture in the following strong quantitative form, obtaining sharp bounds.
If $\delta = \min_{\z}|\z-1|$ over the complex roots $\z$ of $f_X$, and $X^{\ast} := (X-\mu)/\s$,
then 
\[ \sup_{t \in \R} \left|\PP(X^{\ast} \leq t) - \PP( Z \leq t) \, \right|  = O\left(\frac{\log n}{\delta\s} \right)
\] where $Z \sim \NN(0,1)$ is a standard normal. 
This gives the best possible version of a result of Lebowitz, Pittel, Ruelle and Speer.
We also show that if $f_X$ has no roots with small \emph{argument}, then $X$ must be approximately normal, again in a sharp quantitative form: if we set $\delta = \min_{\z}|\arg(\z)|$ then 
\[ \sup_{t \in \R} \left|\PP(X^{\ast} \leq t) - \PP( Z \leq t) \, \right|  = O\left(\frac{1}{\delta\s} \right).
\] Using this result, we answer a question of Ghosh, Liggett and Pemantle by proving a sharp multivariate central limit theorem for 
random variables with real-stable probability generating functions.
\end{abstract}

\maketitle

\section{Introduction}
In his influential paper on negative dependence, Pemantle \cite{pemantle2000} set out a list of desirable combinatorial properties for ``the correct'' definition of negatively dependent random variables and laid out a number of natural conjectures. In their celebrated paper, Borcea, Br\"{a}nd\'{e}n and Liggett \cite{BBL} provided such a definition by making a striking connection with the blossoming subject of real-stable polynomials; it turns out that the definition that Pemantle sought is best described in terms of the zeros of the associated probability generating function. For this, let $X \in \{0,\ldots,n\}^d$ be a random variable\footnote{Throughout the paper we will slightly abuse notation and write $X \in S$, for a random variable $X$ and a set $S$, as shorthand for ``$X$ takes values in the set $S$.''}, let 
\[ f_X(z_1,\ldots,z_d) := \sum_{(i_1,\ldots,i_d) } \PP( X = (i_1,\ldots,i_d))z^{i_1}\cdots z^{i_d},\] 
 be its probability generating function and define $$\bH  = \{ (z_1,\ldots,z_d) \in \C^d : \im(z_i) > 0, \textit{ for all } i \}$$ to be the \emph{upper half-plane}. 
A polynomial $f \in \R[z_1,\ldots,z_d]$ is said to be \emph{real-stable} if it has no roots $\z$ in the upper half-plane $\bH$ and a random variable $X$ is said to be \emph{strong Rayleigh} if its probability generating function $f_X$ is real-stable.
Borcea, Br\"{a}nd\'en and Liggett showed that strong Rayleigh random variables admit a natural theory of negative dependence and provided many 
natural examples of strong Rayleigh distributions: spanning tree distributions, uniform random matching distributions in graphs and determinantal measures. In the years since this notion has been well studied and many further examples have been found \cite{pemantle-survey,pemantle-peres,liggettstability,sample-Strong-R,kyng2018,OG-TSP,random-Geo,li2017polynomial}.

In addition to the connection with negative dependence, the theory of real-stable polynomials has had many recent successes,
notably Borcea and Br\"{a}nd\'{e}n's \cite{BorBranI,BorBranII} powerful classification of linear operators that preserve real stability; its role in Marcus, Spielman, and Srivastava's spectacular proof of the Kadison-Singer conjecture \cite{InterlacingII}; and in Gurvits's surprising and simple proof of (and extensions of) the van der Waerden conjecture \cite{Gurvits,Gurvits08}; among others \cite{InterlacingI,AOSS18,kyng2019}.

In this paper, one of our main motivations is to finish a program set in motion by Ghosh, Liggett and Pemantle \cite{GLP} to show that if $X_n \in \{0,\ldots,n\}^d$ 
is a sequence of random variables, with real stable generating functions, then $X_n$ tends to a multivariate Gaussian distribution, after centering and scaling, provided $\s_n \rightarrow \infty$. We will derive this theorem by first proving results on univariate polynomials, with much looser restrictions on the roots, and then ``lifting'' these results to the multivariate setting. 

In the univariate setting, work on the connection between roots of polynomials and their coefficients reaches back (at least) to Cauchy's quantitative work on the fundamental theorem of algebra \cite{cauchy1828exercises}, but was perhaps first intensely studied by Littlewood and Offord \cite{LittlewoodOfford,LittlewoodOfford-II,LittlewoodOfford-III}, Szeg\H{o} \cite{Szego}, Bloch and P\'{o}yla \cite{BlochPolya} and Schur \cite{Schur} among others (see \cite{clt1} for more discussion). 
To give a bit of flavor of these results, we mention only one such result from this vast literature that is most relevant for us here. In 1950, Erd\H{o}s and Tur\'{a}n \cite{erdos-turan} proved that if $P(z) = \sum_{k=1}^n a_kz^k$ is a polynomial ($a_0a_n\not=0$) with sufficiently ``flat'' coefficients, meaning $(|a_0||a_n|)^{-1/2}\sum_{k=1}^n |a_k| = e^{o(n)}$, then the roots of $P$ are approximately radially-equidistributed in the complex plane, meaning that each sector $\{z : \alpha \leq \arg(z)\leq \beta \}$, for $0 \leq \alpha <\beta \leq 2\pi$ contains roughly $n(\beta-\alpha)/2\pi$ roots. This result has been adapted to different settings \cite{erdos-polys} and generalized and sharpened several times \cite{Amoroso,Mignotte,Bilu}. For more details, we refer the reader to the lovely articles of Granville \cite{Granville} and Soundararajan \cite{Sound}.

In this paper, we show that a substantial amount of information about the coefficients of a polynomial can be derived from its locus of zeros, if we additionally assume the polynomial is a probability generating function, which is to say, it has non-negative coefficients. 
A surprising first step in this direction is due to Lebowitz, Pittel, Ruelle and Speer \cite{LPRS}, who showed that if, for each $n\geq 1$, 
$X_n \in \{0,\ldots,n\}$ is a random variable for which $f_{X_n}$ has no zeros in a neighborhood of $1 \in \C$, and $\s_n n^{-1/3} \rightarrow \infty$ then $(X_n-\mu_n)\s_n^{-1}$ tends weakly to a normal distribution (see also the 1979 work of Iagolnitzer and Souillard in the context of the Ising model \cite{ising1979}). Inspired by this advance, Pemantle \cite{clt1,pemantle}, was lead to conjecture that the variance condition in the theorem of Lebowitz, Pittel, Ruelle and Speer could be greatly improved.

\begin{conjecture}[\hspace{1sp}\cite{pemantle}] For $\delta >0$ and each $n \geq 1$, let $X_n \in \{0,\ldots,n\}$ be a random variable with mean $\mu_n$, standard deviation $\s_n$ and for which the the roots $\z$ of the probability generating function $f_{X_n}$ satisfy $|\z-1| \geq \delta$. Then $(X_n-\mu_n)\s_n^{-1} \rightarrow N(0,1)$, provided $\s_n \rightarrow \infty$.
\end{conjecture}

In recent work, the authors \cite{clt1} refuted this conjecture by showing that for any $C>0$ there exist random variables $X_n \in \{0,\ldots,n\}$
with $\s_n > C\log n$ that \emph{are not} asymptotically normal and for which $f_{X_n}$ has no roots in a neighborhood of $1 \in \C$. On the other hand, the authors also showed that Pemantle was right to suspect that the variance condition in the work of Lebowitz, Pittel, Ruelle and Speer could be significantly improved, by showing that it is sufficient to assume $\s_n > n^{\eps}$, for any $\eps >0$. 

Here, we completely resolve the conjecture of Pemantle, by showing $\s_n(\log n)^{-1} \rightarrow \infty$ is sufficient to guarantee convergence to a normal distribution. In fact, we prove a sharp quantitative version of this theorem that gives an optimal bound on the maximum discrepancy between a random variable $X$ and a standard normal, based only on the distance of the closest root of $f_X$ to $1 \in \C$.

\begin{theorem} \label{thm:logn}
Let $X \in \{0,\ldots,n\}$ be a random variable with mean $\mu$, standard deviation $\s$ and probability generating function $f_X$ and set $X^{\ast} = (X-\mu)\s^{-1}$.
If $\delta \in (0,1)$ is such that $|1-\z| \geq \delta $ for all roots $\z$ of $f_X$ then\footnote{The implicit constant can be taken to be $2^{3261}$.} 
\begin{equation}\label{equ:lognCloseToNorm} \sup_{t \in \R}| \PP(X^{\ast} \leq t) - \PP(Z \leq t) | =  O\left(\frac{\log n}{\delta \s}\right)\,,
 \end{equation} where $Z \sim N(0,1)$.
\end{theorem}

We note that this immediately implies the following limit theorem for distributions with no roots close to $1 \in \C$.

\begin{corollary} \label{cor:ball-CLT}
For each $n\geq 1$, let $\delta_n \in (0,1)$ and $X_n \in \{0,\ldots,n\}$ be a random variable with mean $\mu_n$, standard deviation $\s_n$ and probability generating function $f_n$.
If $|\z-1| \geq \delta_n$ for all roots $\z$ of $f_n$ and 
\[ \s_n \delta_n (\log n)^{-1} \rightarrow \infty 
\] then $(X_n -\mu_n)\s_n^{-1} \rightarrow N(0,1)$ in distribution. 
\end{corollary}
The condition on the standard deviation $\s_n$ in Corollary~\ref{cor:ball-CLT} is sharp, both in terms of $\delta_n$ and in terms of $n$.

Our second result (and the main ingredient in the proof of the multivariate central limit theorem for strong Rayleigh distributions) says we can weaken the variance condition in Theorem~\ref{thm:logn} all the way to $\s_n \rightarrow \infty$ (the obvious\footnote{The law of $(X-\mu_n)\s_n^{-1}$ is supported on point masses of distance $\geq \s_n$ so we must have $\s_n \rightarrow \infty$ if the sequence approximates the continuous Gaussian distribution.} necessary condition) if we further assume that the sequence $f_n$ has no roots in a small sector $\{ z : |\arg(z)|<\delta\}$ containing the positive real axis. Again, we prove a sharp, quantitative version of this by obtaining an optimal bound on the discrepancy between a normal random variable $Z \sim N(0,1)$ and a random variable $X$, 
based only on the smallest angle made by a root of $f_X$ and the positive real axis.

\begin{theorem} \label{thm:Polysector} 
Let $X \in \{0,\ldots,n\}$ be a random variable with mean $\mu$, standard deviation $\s$ and probability generating function $f_X$ and set $X^{\ast} = (X-\mu)\s^{-1}$.
If $\delta >0$ is the minimum of $|\arg(\z)|$ over the roots $\z$ of $f_X$ then\footnote{The implicit constant can be taken to be $2^{3257}$.}
\begin{equation} \label{equ:sectorCloseToNorm} \sup_{t \in \R}| \PP(X^{\ast} \leq t) - \PP(Z \leq t) | = O\left( \frac{1}{\delta \sigma}\right), \end{equation}
where $Z \sim N(0,1)$. 
\end{theorem} 

Theorem~\ref{thm:Polysector}, immediately implies the following limit theorem for distributions where the smallest argument of a root \emph{just} exceeds the reciprocal of the standard deviation. 

\begin{corollary}\label{cor:sector-CLT} For each $n\geq 1$, let $X_n \in \{0,\ldots, n\}$ be a random variable with mean $\mu_n$ standard deviation $\s_n$ and probability generating function $f_n$. If the roots $\z$ of $f_n$ satisfy $|\arg(\z)| \geq \delta_n$ and \[ \delta_n\s_n \rightarrow \infty,\] then $(X_n - \mu_n)\s_n^{-1} \rightarrow N(0,1),$ in distribution. 
\end{corollary}

Again, as we shall see in Section~\ref{sec:sharpness}, the condition on $\s_n$ is sharp for all sequences $(\delta_n)_n$.

With Theorem~\ref{thm:Polysector} in hand, it is not hard to prove our multivariate central limit theorem for strong Rayleigh distributions, following a key observation of Ghosh, Liggett and Pemantle. To properly state this result, recall that if $X = (X_1,\ldots,X_d) \in \R^d$ is a random variable, its \emph{covariance matrix} $A = A(X)$ is a symmetric, semi-definite matrix defined by 
\[ (A)_{i,j} := \E X_iX_j - \E X_i\, \E X_j ,
\] and its \emph{maximum variance} $\s_n^2$ is defined as the $\ell^2$-operator norm of $A$. For a $d\times d$ positive semi-definite matrix $A$
and $\mu \in \R^d$ we denote the Gaussian random variable with covariance matrix $A$ and mean $\mu$ by $N(\mu,A)$.

Motivated by a vast literature on central limit theorems and multivariate generating functions \cite{bender73,bender83,gao92,pemantle04,pemantle08,Branden-Jonasson},
Ghosh, Liggett and Pemantle proved a central limit theorem for strong Rayleigh distributions, in the case that the sequence maximum variances $\s_n^2$ grows sufficiently quickly, namely $\s_nn^{-1/3} \rightarrow \infty$. In particular, they proved that if $d \in \N$, and $X_n \in \{0,\ldots,n\}^d$ is a sequence of strong Rayleigh distributions with covariance matrices $\{A_n\}$ that satisfy $\s_n^{-2}A_n \rightarrow A$ and $\s_nn^{-1/3} \rightarrow \infty$ then $(X_n -\mu_n)\s_n^{-1} \rightarrow N(0,A)$, weakly. They conclude their paper by asking for the best possible condition on the growth of $\s_n$ and ask, in particular, if the condition $\s_n \rightarrow \infty$ is sufficient in their theorem. 
In \cite{clt1}, we made progress on this problem by showing that $\s_n > n^{\eps}$ is sufficient, for any $\eps >0$. Here, we are able to completely resolve the question
of Ghosh, Liggett and Pemantle by showing that the obvious necessary condition, $\s_n \rightarrow \infty$, is indeed sufficient.

\begin{theorem} \label{thm:stable_clt} For $d \in \N$ and each $n \geq 1$, let $X_n \in \{0,\ldots ,n\}^d$ be a random variable with covariance matrix $A_n$ and maximum variance $\s^2_n$. If the probability generating functions of $X_n$ are real-stable, $\s_n \rightarrow \infty$ and $\s_n^{-2}A_n \rightarrow A$, then 
\[(X_n - \mu_n)\s_n^{-1} \rightarrow N(0,A),\] in distribution. \end{theorem}

We can additionally prove a quantitative form of this theorem (in the spirit of Theorem~\ref{thm:logn} and Theorem~\ref{thm:Polysector}), but we defer this more technical result to a later paper.

It is perhaps interesting to note that in this paper, we make essentially no use of the rich theory of stable polynomials and, as a result,
our work here provides (what appears to be) a new and flexible tool-set for working with real stable polynomials. To illustate, in Section~\ref{sec:CLT-for-stable} we show that our method immediately implies a version of Theorem~\ref{thm:stable_clt} for \emph{Hurwitz stable polynomials} \cite{hurwitz64}, a similar and well studied notion \cite{matroid-halfplane,wagner05,wagner09} along with other polynomials satisfying a similar ``half-plane property''. Our methods are also of use beyond proving central limit theorems. In a forthcoming paper \cite{var-paper}, we use the tools from this paper to prove a close connection 
between the roots of $f_X$ and the \emph{variance} of $X$.

To conclude the introduction, we take a moment to situate our results in a wider mathematical context and draw a few connections. The use of 
roots to study combinatorial distributions has a long and distinguished history in mathematics and has provided many surprising connections.
Perhaps the most classical instance comes from the monumental contribution of Riemann, who connected the location of the roots of the Riemann zeta function 
with the distribution of the primes, thereby setting the stage for the proof of the prime number theorem and, more generally, the field of analytic number theory.
Indeed, the Riemann hypothesis is exactly a statement about the location of the roots of the zeta function. 
In a somewhat different area, statistical physics, Lee and Yang \cite{lee-yang,yang-lee} drew another surprising and influential connection between the roots of polynomials and phase transitions in physical systems. They were also able to show that in the case of the Ising model, the roots of the probability generating function (i.e. the partition function) associated with the number of ``up-spins'' always lie on the unit circle. In combinatorics, the roots of various polynomials associated with graphs and other combinatorial objects have also been shown to have particular regions without zeros (see, for example, \cite{jackson-procacci-sokal,peters-regts17,regts18,chro-poly}). 
The most classical example here is the striking theorem of Heilmann and Lieb \cite{heilmann-lieb72}, which says that the roots of the matching polynomial are \emph{real}. 

In some cases, zero-free regions have been used directly to prove central limit theorems. In the case of matchings, a line of results \cite{harper67,godsil81,rucinski84}, starting with the work of Godsil \cite{godsil81} and cumulating in the work of Kahn \cite{Kahn00}, used the Heilmann Lieb theorem to give general results for when the size of a random matching in a graph is approximately normal. In the case of the Ising model, the work of Lee and Yang was used to prove central limit theorems for the number of ``up-spins'' by Iagolnitzer and Souillard \cite{ising1979} and later by Lebowitz, Pittel, Ruelle and Speer \cite{LPRS}. In a similar vien, Scott and Sokal \cite{scott-sokal05,scott-sokal06}, who built on the work of Shearer \cite{shearer85} and Dobrushin \cite{dobrushin96}, showed a close connection between zero-free regions and the Lovasz Local Lemma, another probabilistic tool which is used for showing that a particular event has non-zero probability. 

The philosophy that appears to have emerged from these advances is that the roots of combinatorially defined objects often have special structure and admit particular zero-free regions. This observation was made explicit by Rota, who sought to give ``combinatorial meaning'' to the distributions of roots in these settings \cite{rota-interview}. 
In this light, one could see our results as a general contribution to this program of Rota (his so-called \emph{critical program}) by giving combinatorial meaning to the roots of a wide class of polynomials.

\section{Outline of Proof}
Theorems~\ref{thm:logn} and Theorem~\ref{thm:Polysector} are proved in parallel and can be thought of as two consequences of the same general method.
As such, in the discussion here, we are intentionally vague about which of these two theorems we are proving.  
Now, let $X \in \{0,\ldots,n\}$ be a random variable with probability generating function $f_X$ and consider the \emph{characteristic function} of $X$, which is a relative of $f_X$ and defined as $\psi_X(\xi) := f_X(e^{i\xi})$, where $\xi \in \R$. The relevant feature of the characteristic function is that it detects the closeness between two probability distributions: a sequence of random variables $Y_n$ converge in distribution to the random variable $Y$ if and only if the sequence of characteristic functions $\psi_{Y_n}$ converge to the characteristic function of $Y$ $\psi_Y$ \emph{point-wise}. Of course, our results here are quantitative, but this fact serves as some guide: to show that $Y$ is approximately normal
it is enough to show that $\psi_{Y}(\xi) \approx e^{-\xi^2/2}$, where $\psi_Z  := e^{-\xi^2/2} $ is the characteristic function
of the standard normal $Z \sim N(0,1)$. With this in mind, it is natural to center and scale $X$, by writing $X^{\ast} := (X-\mu)\s^{-1}$ and then to consider
the logarithm of $\psi_{X^{\ast}}$, due to the exponential form of $\psi_Z$. Indeed, we will be able to express $\log|\psi_{X^{\ast}}(\xi)|$ as  
\begin{equation} \label{eq:sketch-cumulant-sequence} \log|\psi_{X^{\ast}}(\xi)| = \sum_{j\geq 2} a_j\s^{-j}\Re(i^j\xi^j), 
\end{equation} where $\xi$ is in a sufficiently small neighborhood of $0$ and $(a_j)$ is a sequence of real numbers. 

It turns out that $a_2 = -\s^2/2$ and hence the first term of the series is $-\xi^2/2$, just as we saw in the exponent of $\psi_Z$. From this vantage, our task is becomes clear: we need to show that $|a_j| \ll \s^j$ in order to have $\psi_{X^{\ast}} \approx e^{-\xi^2/2}$.

With our goal now laid out, we turn to consider the function $u(z) := \log|f_X(z)|$ in a region around $1 \in \C$ in the complex plane. Note that here we can quite naturally make use of our zero-free hypothesis: if $f_X$ is zero free in a region, then the function $u$ is harmonic in this region. Now, while the fact that $u$ is harmonic on a particular region \emph{is} a useful property, it is far from enough to prove our main theorems\footnote{In fact, as we discussed in \cite{clt1}, the results of
Lebowitz, Pittel, Ruelle and Speer are actually \emph{sharp} if one generalizes their theorem to polynomials that have negative coefficients. Thus, we \emph{must} use the non-negativity hypothesis in an essential way.}; we will additionally need to make particular use of the fact that $f_X$ has positive coefficients, a property that we use in the form of ``weak-positivity'' for the function $u$ (see Section~\ref{sec:defs}).

As we will see in Section~\ref{sec:b-decreasing}, this notion of weak-positivity interacts nicely with the harmonic property of $u$, to give us another ``positivity'' notion which we make heavy use of.  For $b\geq 0$, $\eps >0$, we say that function $u$ on $B(1,\eps)$, with $u(z) = u(\bar{z})$, is $b$-\emph{decreasing} if for all $0 < \t_1 < \t_2 <\delta$ we have 
\begin{equation} \label{eq:outline-b-decreasing} u(\rho e^{i\t_1}) - u(\rho e^{i\t_2}) \geq -b,
\end{equation} where the functions are defined. In Section~\ref{sec:b-decreasing} we prove Lemma~\ref{lem:b-decreasing}, which is our main tool for showing that a function is $b$-decreasing. This lemma says that if $u$ is a weakly-positive, harmonic function on $S := \{z \in \C : R^{-1}< |z|<R,\, |\arg(z)| < \delta \}$
then the function $u$ is $b$-decreasing if 
\begin{equation} \label{eq:outline-cond-b-dec} \exp\left( -\delta^{-1} \log(R/r)\right )\max_{z}  |u(z)|  \leq b/10, \end{equation}
where the maximum is taken over the ``ends'' of $S$, defined as 
\[ S^{\ast} := \{ z \in \C: |z| \in \{R^{-1},R\}, \, \arg(z) \in [-\delta,\delta] \}.\] Without going into details, one can already see two important features of \eqref{eq:outline-cond-b-dec}. Firstly, if $u$ is harmonic and weakly-decreasing in an \emph{entire} sector $\{z : \arg(z) \in [-\delta,\delta] \}$, then the left-hand-side of 
\eqref{eq:outline-cond-b-dec} can be taken to be arbitrarily small (by letting $R \rightarrow \infty$) and so we learn that $u$ is $0$-decreasing, which perhaps should strike the reader as a reasonably strong property. Secondly, we note the exponential dependence on the width $\delta$ of the sector. This ultimately accounts for the factor of $\log n$ that appears in Theorem~\ref{thm:logn}.

With this tool in-place, we turn to show how to use the $b$-decreasing hypothesis to get some control of the sequence $(a_j)_{\geq 2}$. In a series of steps, we work towards Lemma~\ref{thm:maintechnical} and its important corollary, Corollary~\ref{cor:CLT}. These results are perhaps the main technical contributions of this paper and their proof consumes Sections~\ref{sec:Proof-of-tec-thm-key-comparison}-\ref{sec:technical-proof}. 
Lemma~\ref{thm:maintechnical} says that if $u$ is a $b$-decreasing, weakly-positive, harmonic function on $S$ then the associated characteristic function $\psi$ of $X^{\ast}$ must look like 
\begin{equation} \label{eq:outline-thm-maintech} \psi(\xi) = \exp\left( -\xi^2/2 + R(\xi) \right),
\end{equation} where the ``remainder term'' $R(\xi)$ is controlled by $|R(\xi)| \leq c|\xi|^3/(\delta\s)$, for all $\xi \in \R$ satisfying $|\xi| < c_2\delta\s$. 
Of course, the reader should interpret \eqref{eq:outline-thm-maintech} as saying that $\psi$ looks like the characteristic function of a standard normal,
up to the remainder-term $R$.

The proof of \eqref{eq:outline-thm-maintech} is carried out in three main steps. The first step is proved in 
Section~\ref{sec:Proof-of-tec-thm-key-comparison} where we prove an important supporting result, Lemma~\ref{lem:money}, that allows us to compare the maximum of $u_0 := u - \mu\log|z|$, a re-normalized form of $u$,
to a particular function $\varphi_{\g,b}(z)$ (defined in Section~\ref{sec:defs}) which is both harmonic and \emph{positive} on a region containing $1\in \C$.  

In Section~\ref{sec:TailOfcumulants} we use this ``comparison'' lemma to prove Lemma~\ref{lem:CumulantDecay}, which tells us that the sequence $(a_j)_{j\geq 2}$ has nice decay properties; for every $L \geq 2$, we have that 
\begin{equation} \label{eq:outline-CumulantFrac} 
\frac{ \sum_{j\geq L } |a_j|\eps^j }{ \sum_{j\geq 2} |a_j|\eps^i } \leq C \cdot 2^{-L}, \end{equation}
provided 
\[  \sum_{j\geq 2} |a_j|\eps^j > b, \]
where $C$ is a large, but absolute, constant and $\eps \approx \delta$. 

We stress that \eqref{eq:outline-CumulantFrac} is a major step towards proving Lemma~\ref{thm:maintechnical} and indeed Sections \ref{sec:Proof-of-tec-thm-key-comparison} and \ref{sec:TailOfcumulants} are probably the most pivotal in
the paper. However, \eqref{eq:outline-CumulantFrac} is not quite enough. Roughly speaking, \eqref{eq:outline-CumulantFrac} says that we have quite a bit of the ``mass''
of the sequence $(a_j)_j$ is focused on the early terms, but what we actually need to show that ``most'' of the mass is on the \emph{second term}, $a_2 = - \s^2/2$.

For this next step, carried out in Section~\ref{sec:BES}, we prove Lemma~\ref{lem:BES}, which says that if $u$ is weakly-positive and harmonic around $1 \in \C$,
\emph{and} $|a_j|$ is large for some small $j\geq 2$ then $|a_2|$ must also be large. This allows us to control the magnitude of each of the terms $|a_j|$
relative to the value of $|a_2|$. Applying Lemma~\ref{lem:CumulantDecay} and Lemma~\ref{lem:BES} in sequence allows us to deduce \eqref{eq:outline-thm-maintech}.

Now, while \eqref{eq:outline-thm-maintech} tells us that the characteristic function $\psi$ of $X^{\ast}$ is roughly like the characteristic function of a
standard normal, we really care about showing that the \emph{distribution} of $X^{\ast}$ is close to the distribution of a standard normal. For this, we need an appropriate ``Fourier inversion'' step. This step is carried out in Section~\ref{sec:proofs-of-thms}, just before we go on to deduce Theorems~\ref{thm:logn} and \ref{thm:Polysector}.

In Section~\ref{sec:CLT-for-stable}, we turn to use the results developed in previous chapters to prove our multivariate central limit theorem for strong Rayleigh distributions. This is achieved by first using a fundamental observation of Ghosh, Liggett and Pemantle that says that if $X \in \{0,\ldots,n\}^d$ is a random variable with real-stable generating function, then the characteristic functions of the one-dimensional projections $\langle X, v\rangle$, where $v \in \Z_{\geq 0} ^d$, have no roots in a small sector. Theorem~\ref{thm:Polysector} then allows us to show all of these these projections are approximately normal. We then use a strong version of the Cram\'{e}r-Wold theorem to lift this information to deduce that $X$ itself must be a approximately normal. 

 In Section~\ref{sec:sharpness} we give examples, demonstrating the tightness of our results. Finally, in Section~\ref{sec:general-distributions} we
briefly discuss how the main results of this paper can be generalized to go beyond polynomials to prove sharp results for power series and more general analytic functions.

\section{Definitions and basic properties} \label{sec:defs}

In this section we fix a few notations and introduce the central objects of our proof. 

Throughout, we use the notations $\R_{\geq 0}, \R_{\leq 0}$ and so on, to denote the non-negative reals and non-positive reals respectively and extend these definitions in the obvious way to $\Z$. If $z \in \C$, we write $z = re^{i\t}$, where $r >0$ and $\t \in [-\pi,\pi]$, and then define the \emph{argument of} $z$ to be $\arg(z) = \t$.
For $-\pi \leq  \beta < \alpha \leq \pi$, we define the \emph{sector}
\[ S(\alpha,\beta) := \{ z \in \C \setminus \{0\} : \alpha \leq \arg z \leq \beta \} \] and
$S(\alpha) := S(-\alpha, \alpha)$. For $R \geq 1$ and $\eps >0$, we define the \emph{truncated sector}
\[ S_{R}(\eps) := \{ z  \in \C: |z| \in [R^{-1},R],\textit{ and } \arg(z) \in [-\eps,\eps]\} \]
and define 
\[ S^{\ast}_{R}(\eps) := \{ z \in \C: |z| \in \{ R^{-1},R\},\textit{ and } \arg(z) \in [-\eps,\eps]\} 
\] to be the \emph{ends} of the sector $S_R(\eps)$. 
We also use the notation $S_R(\alpha,\beta) = S(\alpha,\beta)\cap S_R(\pi) $ in a similar way
and use the (standard) notation $\p \Omega$, to denote the boundary of a region $\Omega \subseteq \C $.

\subsection{The logarithmic potential}
If $X \in \{0,\ldots,n\}$ is a random variable, define
\[ f_X(z) = \EE\, z^X = \sum_{k=0}^n \PP(X = k) z^k  
\] to be its probability generating function, $\mu = \EE\, X$ for its mean and $\sigma^2 = \Var[X]$ for its variance. Also note 
that $f_X(1) = 1$. Now define 		
\[ u(z) = u_X(z) := \log |f_X(z)|\,, 
\] to be the \emph{logarithmic potential} of $X$ and observe that if $f_X(z)$ is zero-free in an open set $\Omega \subseteq \C$ 
then $u$ is \emph{harmonic} on $\Omega$. This key connection will allow us to exploit the theory of harmonic functions in certain regions of $\C$.
We will say that a function $u$ on $\Omega$ is \emph{symmetric on} $\Omega$ if 
\begin{equation} \label{equ:sym} u(z) = u(\zbar) 
\end{equation} for all $z$ with $z,\bar{z} \in \Omega$.  Of course, the logarithmic potential $u_X(z)$ is symmetric as $f_X$ is a polynomial with real coefficients 
and so 
\[ u(z) = \log|f_X(z)| = \log|\overline{f_X(z)}|= \log|f_X(\bar{z})| = u(\bar{z}).\]
A third key property is particular to 
the fact that $f_X$ is a probability generating function; that is, it is a polynomial with non-negative coefficients. We say that a function $u$
is \emph{weakly-positive} on $\Omega$ if 
\begin{equation} \label{equ:weakPositivity} u(|z|) - u(z) \geq 0,  \end{equation}
for all $z \neq  0$ with $z,|z| \in \Omega$.  Weak-positivity of $u_X$ follows by taking the logarithm of both sides of the inequality 
\[ |f_X(z)| = |\EE z^X | \leq \EE|z|^X = |f_X(|z|)|.\]  

We also note a useful expression of $u_X$ in terms of the roots $\{ \z \}$ of $f_X$
	\begin{equation} \label{equ:rootFormOfu} u_X(z) = \sum_{|\zeta| < 1} \log\left|1 - \frac{\zeta}{z} \right| + \sum_{|\z| \geq 1 }\log\left|1 - \frac{z}{\zeta}\right| + c_X + N_X \log|z| \,, 
	\end{equation} where $c_X$ is defined so that $u_X(1) = \log|f_X(1)| = 0 $ and $N_X$ is the number of roots of $f_X$ with $|\zeta| < 1$.

\subsection{The exponential scale} We shall often work with the function $u = u_X$ on an ``exponential scale'' by defining $ U(w) := u(e^{w})$.  
Note that $U(w)$ is harmonic when $u$ is (in the appropriate domains) and is also symmetric, since 
\[ U(\bar{w}) = u(e^{\bar{w}}) = u(\overline{e^{w}}) = u(e^w) = U(w). \]
 The importance of this form is made clear by Lemma~\ref{lem:USeries}; the Taylor expansion of $U(w)$ at $w = 0 $ reveals the \emph{cumulants} of $X$, which we denote by $(\kappa_j)_{j\geq 1}$. We don't need to draw on much external information here about this important sequence, but we do need to note that the first and second cumulants are familiar probabilistic quantities. Indeed,
 
 \begin{equation}\label{eq:k1-is-mean} \k_1 = \frac{d}{dw} u(e^w) \Big|_{w = 0} = \mu, 
\end{equation} and 
 \begin{equation} \label{eq:k2-is-var} \k_2 = \frac{d^2}{dw^2} u(e^w) \Big|_{w = 0} = \sigma^2\,, 
 \end{equation} which exist under the condition that $u$ is harmonic in a neighborhood of $1 \in \C$. 
 
Since our interest is in central limit theorems, when working with $u$ it is often useful to ``subtract out'' the term corresponding $\mu$. In particular, define
 \begin{equation} \label{eq:def-of-u0} u_0(z) := u(z) - \mu \log |z|\,, \end{equation}
and correspondingly define $U_0(w):= u_0(e^{w})$. If $u$ is an arbitrary function (that is, not necessarily coming from a random variable)
we may define $u_0$ and $U_0$ in the same way, by simply taking $\mu:=\frac{d}{dw} u(e^w)\big|_{w = 0}$, in the case that this quantity exists.

\begin{lemma} \label{lem:USeries} For $\eps \in (0,1/2)$, let $u$ be a symmetric harmonic function on $B(1,2\eps)\subseteq \C$ with $u(1)=0$. 
Then if $U,U_0$ are defined as above, we may express 
\begin{equation} \label{equ:USeries} U(w) = \sum_{j\geq 1} a_j \Re( w^j ) \end{equation} and 
\begin{equation}  \label{equ:U0Series} U_0(w) = \sum_{j\geq 2} a_j\Re(w^j ),\end{equation}
for all $w \in B(0,\eps)$.
Where the $a_j$ are real numbers and $j!a_j = \kappa_j $, where $\k_j$ is the $j$th cumulant. 
\end{lemma}
\begin{proof} First note that $U(0) = u(1) = 0$ and that $U(w)$ is harmonic and symmetric in a ball $B(0,\eps)$, since $|w|<1/2$ and $|e^w-1|\leq 2|w| < 2\eps$.

Now since $U$ is harmonic in $B(0,\eps)$, we may write $U(w) = \Re f(w)$ for a function $f$ which is analytic in $B(0,\eps)$ (see Conway's classic text \cite{conway}, Chapter VIII, Theorem 2.2, p. 202 for a proof). We then express $f$ as a power series to obtain
\begin{equation}\label{eq:series-for-U} U(w) = \sum_{j \geq 0} \Re( a_j  w^j)\,.\end{equation}
	
	 We now write $w = \rho e^{i\t}$, for sufficiently small $\rho >0$, and use the fact that $U(w) = U(\bar{w})$ to obtain
	
	\[ 0 = U(w) - U(\bar{w}) = \sum_{j\geq 0} \Re(a_j\rho^j (e^{ij\t}  - e^{-ij\t})) =  \sum_{j\geq 0} \rho^j \Re(2ia_j)\sin(j\t) \,.
	\] By the uniqueness of trigonometric series, we have that $ \Re(2ia_j) = 0 $, implying that $a_j$ is real, for all $j\geq 0$. So from
	\eqref{eq:series-for-U} and the fact that $U(0) = 0$, we obtain \eqref{equ:USeries}.
		
	To prove the second part of the claim, we simply note that 
	\[ U_0(w) = u_0(e^{w}) = u(e^{w}) - \mu \log|e^{w}| = U(w) - \mu \Re(w).
	\] So (\ref{equ:USeries}) and the fact that $\mu = a_1$ yields (\ref{equ:U0Series}).\end{proof}

\vspace{4mm}

Throughout the paper we work with the sequence $(a_j)_{j\geq 1}$ rather than the cumulant sequence $(\kappa_j)_{j\geq 1}$.
We call the sequence $(a_j)_{j\geq 1}$ the \emph{normalized cumulant sequence}.

\subsection{Two important ``difference functions'' and $b$-decreasing}
For $\tau \in (0,\pi)$ and $b \geq 0$, we define the function $h_{\tau}$ that compares values of $u$ reflected about the line $\{te^{i\tau/2}: t \geq 0 \}$
\[ h_{\tau,b} (z) := u(z) - u(e^{i\tau} \bar{z} ) + b. 
\] We also define, for $\g \in (0,\pi)$, the closely related function
\[ \vp_{\g,b}(z) := u(z) - u(e^{i\g}z) + b. 
\] 

We first observe that if $u$ is harmonic in a sector, then $h_{\tau,b}$ and $\vp_{\g,b}$ are harmonic in a slightly smaller sector. 

\begin{lemma}\label{lem:h-phi-HarmInSector} For $R \in (1,\infty]$, $\delta >0$, let $u$ be a harmonic function in\footnote{Here we understand $S_{\infty}(\delta)$ to mean the sector $S(\delta)$.} $S_R(\delta)$. Then
for any $\tau,\g \in (0, \delta/2)$ we have that  
the functions $h_{\tau,b}$ and $\vp_{\g,b}$ are harmonic in the sector $S_R(\delta / 2)$.\end{lemma}
\begin{proof}
	We note that if $v$ is harmonic in $\Omega$ and $\beta \in \C \setminus \{ 0 \}$ then $v(\beta z)$ is harmonic in $\beta^{-1}\Omega$ and $v(\bar{z})$ is harmonic 
	in $\{ \bar{z} : z \in \Omega \}$. Also if $v_1,v_2$ are harmonic in $\Omega_1,\Omega_2$ respectively, then $v_1 - v_2$ is harmonic in $\Omega_1 \cap \Omega_2$
	(see \cite{harmonic-axler}, Chapter 1).
	Thus, $u(\bar{z})$ is harmonic in $S_R(\delta)$ and $u(\alpha \bar{z})$ is harmonic in $e^{i\tau}S_R(\delta) = S_R(-\delta +\tau,\delta+\tau)$ and therefore 
	$h_{\tau,b}$ is harmonic in $ S_R(-\delta +\tau,\delta+\tau) \cap S_R(\delta)  \supseteq S_R(\delta/2)$. Likewise, $\vp_{\g,b}$ is harmonic in 
	$S_R(\delta - \g,\delta+\g) \cap S_R(\delta) \supseteq S_R(\delta/2)$.
\end{proof}

\vspace{4mm}

We now arrive at an essential definition, which we have already mentioned in the overview of the proof: $b$-decreasing. For $b \geq 0$ and $\Omega \subseteq \C$, we say a function $u$ is \emph{$b$-decreasing} in $\Omega$ if $$u(\rho e^{i\theta_1}) - u(\rho e^{i \theta_2}) + b \geq 0$$
for all $0 \leq \theta_1 \leq \theta_2 \leq \pi$ with $\rho e^{i\theta_1}, \rho e^{i\theta_2} \in \Omega$. One nice feature of this definition is that $u$ is $b$-decreasing in $\Omega$ if and only if $u_0$ is, since 
\[ u_0(\rho e^{i\theta_1}) - u_0(\rho e^{i \theta_2}) + b = u(\rho e^{i\theta_1}) - u(\rho e^{i\theta_2}) + b .
\] The main motivation behind this definition is easy: it is the correct definition to guarantee the functions $h_{\tau,b}$ and $\vp_{\g,b}$ are positive, for all reasonable choices of $\tau$ and $\g$. Later we shall make heavy use of this fact by way of the so-called Harnack inequalities.  These inequalities guarantee that $h$ and $\vp$ don't vary too much on a set $\Omega$ away from its boundary.

\begin{lemma}\label{lem:h-phi-+veInSector} For $\delta >0$ and $b\geq 0$, let $\tau,\g \in (0, \delta/2)$ and let $u$ be $b$-decreasing and symmetric in $S_{R}(\delta)$. Then $h_{\tau,b}(z) \geq 0 $ for all $z \in S_R(\tau/2)$ and $\vp_{\g,b}(z) \geq 0$ for all $z \in S_R(-\g/2,\delta-\g)$. 
\end{lemma}
\begin{proof}
	Let $z = r e^{i\t} \in S_R(\tau/2)$ with $|\t|\leq \tau/2$.  Then $$h_{\tau,b}( r e^{i\t}) = u(r e^{i\t}) - u(r e^{i (\tau - \t)}) + b\,.$$
	
	In the case $\t \in [0,\tau/2]$, we have that $\theta \leq \tau - \t$, and thus non-negativity of $h_{\tau,b}(r e^{i\t})$ follows from the $b$-decreasing assumption on $u$.  On the other hand, if $\t \in [-\tau/2,0]$, we use symmetry to write 
	$$h_{\tau,b}( r e^{i\t}) = u(r e^{i\t}) - u(r e^{i (\tau - \t)}) + b =   u(r e^{-i\t}) - u(r e^{i (\tau - \t)}) + b$$
	with $0 \leq - \t \leq \tau - \t \leq 2\tau \leq \delta$; non-negativity again follows from $b$-decreasing. The proof for $\varphi_{\gamma,b}$ is similar.
\end{proof}

It will be important for us that Lemmas~\ref{lem:h-phi-HarmInSector} and \ref{lem:h-phi-+veInSector} tell us that 
both $h_{\g,b}$ and $\vp_{\g,b}$ are harmonic and positive in a sector that contains the positive real axis with room to spare on both sides. This means that we can work near $1 \in \C$, without getting too close to the boundary.

\section{Weakly positive and harmonic implies $b$-decreasing } \label{sec:b-decreasing}

In this section we prove Lemma~\ref{lem:b-decreasing}, our main tool for showing that a function is $b$-decreasing.

\begin{lemma} \label{lem:b-decreasing} For $\delta \in (0,\pi)$ and $R > r > 0$ the following hold. Let $u$ be a weakly-positive, symmetric, harmonic function on a neighborhood of $S_R(\delta)$. If $b \geq 0$ is such that 
\begin{equation} \label{equ:bcondition}\left(\frac{r}{R} \right)^{1/\delta} \max_{z \in S_R^{\ast}(\delta) } |u(z)| \leq 3b/8, \end{equation}
then $u$ is $b$--decreasing on $S_{r}(\delta/2)$. 
 \end{lemma}

To prove Lemma~\ref{lem:b-decreasing}, we use a well-known connection between harmonic functions and Brownian motion. The following theorem is a special case of Theorem 3.12 from the book of M\"{o}rters and Peres \cite{peres} and shows how Brownian motion can be used to recover a harmonic function from its boundary values.

\begin{theorem}[Theorem 3.12 of \cite{peres}]\label{thm:harmonic-and-b-motion}
Let $v$ be a function which is harmonic on a bounded, convex set $\Omega \subseteq \C$ and continuous on $\p \Omega$, let $z \in \Omega$ and let $(B_t)_{t\geq 0}$ be a Brownian motion started at $z$. If we define the stopping time $\tau := \min \{ t : B_t \in \partial \Omega \}$ then we have 
\[ v(z) = \EE\, v(B_{\tau}). \] \end{theorem}

\vspace{4mm}

In what follows, we will understand $\tau$ to be the stopping time of a Brownian motion hitting the boundary of $\Omega$, 
$\tau := \min\{ t : B_t \in \p \Omega \}$, unless otherwise stated.

\subsection{A calculation for Lemma~\ref{lem:b-decreasing}}
To Prove Lemma~\ref{lem:b-decreasing}, we need an estimate on the probability that a Brownian motion hits one of the ends of a sector $S_R(\delta)$ before
hitting the sides.

\begin{lemma} \label{lem:poisson-sector} For $\delta \in (0,\pi)$ and $R > r >0$, let $z \in S_r(\delta)$ and $(B_t)_{t\geq 0}$ be a Brownian motion started at 
$z$ and stopped when it hits $\p S_{R}(\delta)$. We have 
\begin{equation} \label{eq:poisson-sector}
\PP\left( B_{\tau} \in S^{\ast}_R(\delta) \right)  \leq  \frac{4}{3} \left(\frac{r}{R} \right)^{c/\delta}\,,
\end{equation} where $c = \log 4/3$.
\end{lemma} 

We should mention that there is an exact formula for the probability in \eqref{eq:poisson-sector} and a proof of this can be found Theorem 7.25 in \cite{peres}.
Here we include a simple proof of this weaker result, which is still sharp up to the constant $c$, to give the reader an intuitive feel for why we have the
exponential dependence in \eqref{eq:poisson-sector} on $\delta$. This dependence is quite important and ultimately explains the logarithmic factor that 
appears in Theorem~\ref{thm:logn} and Corollary~\ref{cor:ball-CLT}. 

Our first step towards Lemma~\ref{lem:poisson-sector} is to study a similar situation in a square. We shall then extend this to rectangles, and then use the conformal invariance of Brownian motion to finish the proof for sectors. 

\begin{observation}\label{obs:poisson-square}
For $\delta >0$ and $y \in [-\delta,\delta]$, let $E_y$ be the event that a Brownian motion, started at $iy \in \C$,
hits either the left of right edges of the square $S := \{ z : \Re(z) \leq \delta, \im(z) \leq \delta \}$ before the top or bottom edges. Then $\PP(E_y) \leq 3/4$.
\end{observation}
\begin{proof}
First, for $y \in [0,\delta] $, we consider the event $H_y$; that a Brownian motion, started at $iy$, hits the top edge of $S$ before hitting any other edge. We claim 
$\PP(H_y) \geq 1/4$. If $y=0$ the result is clear by symmetry. If $y >0$, then we couple the Brownian motion $(B_t)_t$ started at $0$ with a Brownian motion $(B_t+iy)_t$ started at $iy \in S$: clearly $B_t+iy$ will hit the top edge of $S$ on every trajectory that $B_t$ does. So $\PP(H_y) \geq \PP(H_0) =1/4$. Now, turning to $E_y$, simply note that if $y \geq 0$ then $\PP(E_y) \leq 1-\PP(H_y) \leq 3/4$. The case $y<0$, follows by symmetry.\end{proof}

\vspace{4mm}

It is now easy to deduce a version of Lemma~\ref{lem:poisson-sector} for rectangles. Here we see quite naturally where the exponential dependence on $\delta$ appears.

\begin{lemma}\label{lem:poisson-rectangle}
	For $\delta > 0$ and $b > a > 0$, let $Q:=\{ z : |\Re(z)| < b, |\im(z)| < \delta \}$, let $z \in Q$ and let $(B_t)_{t\geq 0}$ be a Brownian motion started at $z$
	which is stopped when it hits $\p Q$. We have \begin{equation}\label{eq:poisson-rectangle}
	\PP \left( B_{\tau} \in Q^{\ast} \right) \leq \exp\left(-c\left\lfloor \delta^{-1}(b-a) \right\rfloor \right), \end{equation}
	where $Q^{\ast} = \{ z : |\Re(z)| = b, |\im(z)| < \delta \}$ and $c = \log 4/3$.
\end{lemma}
\begin{proof}
	 Let $E_y$ be the event defined in Observation~\ref{obs:poisson-square} and let $S(z')$ be the event that a Brownian motion, started at $z' \in Q$, hits one of the lines $\{z' + \delta + it \}_{t\in \R}$, $\{z' - \delta + it \}_{t\in \R}$ before hitting the top or bottom of $Q$. Clearly $\PP(S(z')) = \PP(E_{y})$, where $z' = x+iy$. 
	 
We now connect the rectangle-crossing to the box-crossings; a path that hits one of the ends before hitting the top or bottom of $Q$ must cross at least $\ell \geq \lfloor (b-a)/\delta \rfloor$ boxes without hitting the top or bottom on $Q$. Therefore we have 

\begin{equation}\label{eq:box-bound} \PP\left( B_{\tau} \in Q^{\ast} \right) \leq \sup_{z_1,\ldots,z_{\ell}} \PP\left( S(z_1) \cap \cdots \cap S(z_{\ell})\right) \leq \left( \sup_{y} \PP(E_y) \right)^{\ell}, 
\end{equation} where the supremum is over all complex numbers contained in $Q$ satisfying $z_1 = z$ and $|\Re(z_i) - \Re(z_{i+1})| = \delta$ and the second inequality in \eqref{eq:box-bound} follows from the Markov property of Brownian motion.  
Finally, the result follows by applying Observation~\ref{obs:poisson-square}. 
\end{proof}

\vspace{4mm}

To prove Lemma \ref{lem:poisson-sector} we simply use an analytic map to transform the rectangle $Q$ into the truncated sector $S_{R}(\delta)$. The conformal invariance 
of Brownian motion allows us to finish. To state this property of Brownian motion a little more carefully, let $\phi : \C \rightarrow \C$ be an analytic function 
and let $(W_t)_{t}$ be a Brownian motion in $\C$. The conformal invariance of Brownian motion means that $\phi(W_t)$ traces the path of a Brownian motion, at (possibly) a different speed. In other words, there exists a increasing function $\g : \R_{\geq 0} \rightarrow \R_{\geq 0}$, a Brownian motion $(B_t)$ and a coupling of $B_t$ and $W_t$ for which $\phi(W_t) = B_{\g(t)}$. See the book of M\"{o}rters and Peres \cite{peres}, Theorem 7.20, for a proof.

\vspace{4mm}

\noindent\emph{Proof of Lemma \ref{lem:poisson-sector}:}  
Set $b = \log(R)$, $a = \log(r)$ and observe that the analytic function $\phi(z) := e^z$ maps the rectangle $Q = \{z :  |\Re(z)| < b, | \im(z)|< \delta \}$ to the truncated sector $S_{e^b}(\delta) = S_R(\delta)$; maps $Q = \{z :  |\Re(z)| < a, | \im(z)|< \delta \}$ to $S_r(\delta)$; and maps the ends $R^{\ast} = \{z :  |\Re(z)| = b , | \im(z)|< \delta\}$ to the ends $S_R^{\ast}(\delta)$. 

To finish, choose $w \in \{ s : |\im(s)| < \delta , |\Re(s)| < b ,  \}$ so that $\phi(w) = z$, let $(W_t)_{t\geq 0}$ be a Brownian motion started at $w \in R$
and let $(B_t)_t$ be a Brownian motion started at $z \in S_R(\delta)$, and let $\tau'$ be the stopping time $\tau' := \min\{ t : W_t \in \p Q \}$.
By conformal invariance, there is a coupling of $(B_t)_{t\geq 0}$ and $(\phi(W_t))_t$ so that they trace the same path. It follows that
\[ \PP( B_{\tau} \in S_{R}^{\ast}(\delta)) = \PP( \phi(W_{\tau'}) \in S_{R}^{\ast}(\delta)) = \PP( W_{\tau'} \in Q^{\ast} )
 \leq \exp\left(- (\log 4/3) \lfloor \delta^{-1} \log R/r \rfloor \right), 
\] where the inequality follows from an application of Lemma~\ref{lem:poisson-rectangle}.  Utilizing $\lfloor x \rfloor \geq x - 1$ completes the proof. \qed

 \subsection{The proof of Lemma~\ref{lem:b-decreasing}} We now turn to the heart of Section~\ref{sec:b-decreasing}, Lemma~\ref{lem:IntFormForDiff}.

\begin{lemma}\label{lem:IntFormForDiff} For $\delta >0$, $R >0$, put $\a = e^{i\delta}$ and let $u$ be a weakly-positive harmonic function on a neighborhood of $S_R(0,\delta)$, let $z \in S_R(0,\delta/2)$ and let $(B_t)_{t\geq 0}$ be a Brownian motion started at $z$ and stopped when it hits $\p S_R(0,\delta/2)$ then 
\begin{equation} u(z) - u(\alpha \bar{z})  \geq -2\PP\left(B_{\tau} \in S^{\ast}_R(0,\delta/2)\right)\max_{z \in S^{\ast}_{R}(\delta)}|u(z)|.\end{equation}
 \end{lemma}
\begin{proof}
We define two coupled Brownian motions starting at $z$ and $z^{\circ} := \alpha\bar{z}$, respectively. 
First, let $(B_t)$ be a Brownian motion started at $z$ and, in preparation for defining our Brownian motion started at $z^{\circ}$, we define two stopping times: $\tau = \tau_1 := \min \{ t : B_t \in \p S_R(0,\delta) \}$ and 
$\tau_2 := \min \{ t : B_t \in \p S_R(0,\delta/2) \}$.
Now, define the path $(B^{\circ}_t)_{t\geq 0}$ by $B_{t}^{\circ} := \alpha \bar{B}_t$ for $t\leq \tau_2$ and then $B_t^{\circ} := B_t$ for $t\geq \tau_2$. We now note that $(B_t^{\circ})_t$ is in fact a Brownian motion started at $z^{\circ}$; this is because it is a Brownian motion, by definition, after time $\tau_2$ and it is a reflection of a Brownian motion before time $\tau_2$, which is a Brownian motion. 
The only thing to note is that $B_{\tau_2} = \alpha\bar{B}_{\tau_2}$, that is, the two trajectories agree at $\tau_2$, and thus the whole trajectory is a Brownian motion by the strong Markov property.
Also note that $\tau =\min\{ t : B_t^{\circ} \in \p B_{R}(0,\delta) \}$, by symmetry.

We now apply Theorem~\ref{thm:harmonic-and-b-motion} to $u$, and $z,z^{\circ}$ in the region $S(0,\delta)$ to express 
\[ u(z) = \EE\, u(B_{\tau_1})\, \, \, \textit{ and }\, \, \, u(z^{\circ}) = \EE\, u(B_{\tau_1}^{\circ})
\] and therefore,
\begin{equation} \label{eq:BM-exp-of-u} u(z) - u(z^{\circ}) = \EE\left( u(B_{\tau}) - u(B_{\tau}^{\circ}) \right). 
\end{equation} To evaluate this expectation, we break up the space of trajectories into three events.
\begin{enumerate}
\item $E_1 :=  \{ \arg(B_{\tau_2}) = \delta/2 \}$, the event that $B_t$, $B_t^{\circ}$ meet;
\item $E_2 := \{ \arg(B_{\tau_2}) = 0 \}$, the event that $B_t$ hits $\R_{\geq 0}$, before meeting its reflection; 
\item $E_3 := \{ B_{\tau_2} \in S^{\ast}_R(\delta/2) \}$, the event that $B_t$ hits one of the ends of the sectors, before meeting its reflection.
\end{enumerate}
In the event of $E_1$, we have that $B_t,B_t^{\circ}$ meet before time $\tau$, and therefore $u(B_{\tau}) = u(B_{\tau}^{\circ})$ so 
\begin{equation} \label{eq:I1} I_1 := \EE\,\left( u(B_{\tau})- u(B_{\tau}^{\circ})\right) \mathbbm{1}(E_1) = 0 . \end{equation}
In the event of $E_2$, $B_{\tau} \in \R_{\geq 0}$ and thus $ B^{\circ}_{\tau} = \alpha B_{\tau}$ so $u(B_{\tau}) - u(B_{\tau}^{\circ}) \geq 0 $, by weak-positivity.
In particular,
\begin{equation} \label{eq:I2} I_2 := \EE\left( u(B_{\tau}) - u(\alpha B_{\tau}) \right) \mathbbm{1}(E_2) \geq 0. 
\end{equation}
In the case of $E_3$, we crudely estimate 
\begin{equation} \label{eq:I3} I_3 :=  \EE\, \left( u(B_{\tau}) - u(B_{\tau}^{\circ})\right) \mathbbm{1}(E_3) \geq -2 \PP(E_3) \max_{z \in S_{R}^{\ast}(\delta)} |u(z)|. \end{equation}
Now, from \eqref{eq:BM-exp-of-u}, and \eqref{eq:I1},\eqref{eq:I2},\eqref{eq:I3}, we have
\[  u(z) - u(z^{\circ}) =  I_1 + I_2+ I_3  \geq -2\max_{z \in S_{R}^{\ast}(\delta)} |u(z)| \PP_z\left(B_{\tau} \in S^{\ast}_R(0,\delta/2)\right), \]
as desired.
\end{proof} 
 
\vspace{4mm}

We are now able to prove Lemma~\ref{lem:b-decreasing}.

\vspace{4mm}

\noindent \emph{Proof of Lemma~\ref{lem:b-decreasing}.} 
To show that $u$ is $b$-decreasing on $S_r(\delta/2)$ we let  $\rho \in[1/r,r]$ and let $\t_1,\t_2 \in (0,\delta/2)$ satisfy $\t_2 > \t_1$. We need to show that 
\[ u(\rho e^{i\t_1}) -  u(\rho e^{i\t_2}) \geq -b. 
\] For this, let us put $\phi = \t_1 + \t_2$ and note that $\phi < \delta$.  Set $z = \rho e^{i\t_1}$ and $\a = e^{i\phi}$ and note that $\a \zbar = \rho e^{i\t_2}$.
	We now apply Lemma~\ref{lem:IntFormForDiff} with $\delta = \phi$ to obtain 
\begin{align*} \label{eq:b-decreasing-calc} 
	u(\rho e^{i\t_1}) - u(\rho e^{i\t_2} ) &\geq -2\PP\left(B_{\tau} \in S^{\ast}_R(0,\phi /2)\right)\max_{z \in S^{\ast}_{R}(\phi)}|u(z)|  \\
&\geq - 2 \cdot \frac{4}{3} \left(\frac{r}{R} \right)^{4c/\delta}\cdot \max_{z \in S_R^{\ast}(\delta) } |u(z)| \\
&\geq -b,
\end{align*} where the penultimate inequality holds due to the fact that $S_R^{\ast}(0,\phi/2) $ is a sector of width at most $\delta/2$ and
so we apply Lemma~\ref{lem:poisson-sector} to the sector $S_R(\delta/4)$. The last inequality holds by the condition on $b$ and the inequality $4 \log(4/3) > 1$.
Hence we have shown that $u$ is $b$-decreasing in $S_{r}(0,\delta/2)$. \qed

\section{A key comparison} \label{sec:Proof-of-tec-thm-key-comparison}

With our main positivity-hypothesis in place, we now start with the first in a series of steps to prove Lemma~\ref{thm:maintechnical}. In Sections~\ref{sec:Proof-of-tec-thm-key-comparison}, \ref{sec:TailOfcumulants} and \ref{sec:BES} we build up the ingredients for the proof of Lemma~\ref{thm:maintechnical}, finally stated and proved in Section~\ref{sec:technical-proof}. The objective of this section is to prove Lemma~\ref{lem:money}, which says that under the hypotheses of Lemma~\ref{thm:maintechnical} we can bound the maximum value of $u_0$ in a small box around $1$ in terms 
of the much more amenable function $\vp_{\g,b}$.

\begin{lemma} \label{lem:money} For $b \geq 0$, $\eps \in (0,1/8)$ and $\eta \in (0,\eps]$, let $u(z)$ be a
$b$-decreasing, weakly-positive, symmetric and harmonic function on $B(1,8\eps)$, for which $u(1) = 0$.
Let $u_0$ and $\vp_{\eta,b}$ be the associated functions defined in Section~\ref{sec:defs}. 
We have that 
\[ \max_{z \in B(1,\eps) } |u_0(z)| \leq 3^4 \cdot 3^{96 \eps/ \eta}\vp_{\eta,b}(1). \]
\end{lemma}

To prove this lemma we make a few preparations.

\subsection{Positive on the real line}
In our first step towards Lemma~\ref{lem:money} we show that the function 
\[ u_0(z) = u(z)  - \mu \log|z|, \] is positive on the positive real axis. To prove this, we first need the
following basic fact, which first appears in the work of De Angelis \cite{angelis} and then was slightly\footnote{De Angelis actually assumes that $u = \log|p|$ for a polynomial $p$.} extended in the work of Bergweiler, Eremenko and Sokal \cite{BES},\cite{BE}. We include a short proof.

\begin{lemma} \label{lem:varNonNeg} Let $r > 0$.  If $u$ is weakly positive, symmetric and harmonic on a neighborhood of $r$ then  $U''(\log r) \geq 0$.
\end{lemma} 
\begin{proof}
	Write $z = x + iy$ and put $V(x,y) = U(x + iy)$. Note that the harmonicity of $U$ at $a := \log r$ implies $V_{xx}(a,0) + V_{yy}(a,0) = 0\,.$
	Weak positivity and symmetry of $u$ implies that $V_{yy}(a,0) \leq 0$, since $V_{yy}(a,0)$ can be written
	\[ \lim_{h \rightarrow 0} \frac{ V(a,h) + V(a,-h) - 2V(a,0) }{h^2} = \lim_{h \rightarrow 0} \frac{ u(re^{ih}) + u(re^{-ih}) - 2u(r) }{h^2} \leq 0. \] 
	Thus $$U''(a) = V_{xx}(a,0) \geq 0\,.$$
\end{proof}

\begin{remark}
In the case $u = \log|f_X|$, there is a natural probabilistic interpretation of Lemma~\ref{lem:varNonNeg} that can be turned into a proof. After unwinding the definitions a little, one can see that Lemma~\ref{lem:varNonNeg}
simply says that the random variable defined by the probability generating function $f_X(rz)f^{-1}_X(r)$, has \emph{non-negative} variance. This is trivially true, as all random variables have non-negative variance. However, we have elected to include this more general result, as it will simplify our exposition.
\end{remark}

We now deduce the following small but crucial ingredient in the proof of Lemma~\ref{lem:money}: $u_0$ is non-negative on the positive real axis.

\begin{lemma} \label{lem:u0-positive}
For $\eps \in (0,1)$, let $E \subseteq \C$ be an open set containing the interval $[1-\eps,1+\eps]$. 
If $u$ is weakly positive, harmonic and symmetric in $E$ then  $u_0(r) \geq u(1)$, for all $r \in (1-\eps,1+\eps)$.  
\end{lemma}
\begin{proof}
We may write $r = e^t$ for some $t \in \R$ and apply Taylor's theorem to $U(t)$ at $t =0$ to obtain
\[ U(t) - U(0) - t U'(0) = \frac{t_0^2}{2} U''(t_0), 
\] for some $t_0$ with $t_0 \in (1-\eps,1+\eps)$. Since $U_0(t) = U(t) - tU'(0)$ (\eqref{equ:U0Series}, in Lemma~\ref{lem:USeries}), we have 
 $U_0(t) - U(0)= \frac{t_0^2}{2} U''(t_0) \geq 0$, by Lemma \ref{lem:u0-positive} and so 
\[ u_0(r) = u_0(e^t) = U_0(t) \geq u(1),\] as desired. 
\end{proof}

\subsection{The Poisson density and Harnack inequalities}
In Section~\ref{sec:b-decreasing}, we saw that we could use Brownian motion to recover the values of a harmonic function on $\Omega$ from the values of its boundary $\p \Omega$, by using Theorem~\ref{thm:harmonic-and-b-motion}. In particular, we had that
\begin{equation}\label{eq:BM-exp}
u(z) = \EE\, u(B_{\tau}),
\end{equation} where $(B_t)_{t\geq0}$ is a Brownian motion started at $z \in \Omega$ and $\tau$ is the stopping time of hitting $\p \Omega$.
In the case that $\Omega$ is a ball $B$, the expectation in \eqref{eq:BM-exp} has a density function $P_z$ (with respect to the Lebesgue measure on the circle), and so we can express 
\begin{equation} \label{eq:v-po-int} v(z) = \int_{\partial \Omega} v(s)P_z(s)\, ds.
\end{equation} We define the collection of functions $\{P_z\}$ as the \emph{Poisson density} of the circle\footnote{This definition is designed to parallel the (standard) definition of the \emph{Poisson kernel} which is defined with respect to the complex line integral, rather than the uniform measure on the circle.}
and note that $P_z$ are non-negative functions of the boundary. We refer the reader to \cite{harmonic-axler} and to \cite{conway} for a general treatment of harmonic functions on a ball.

We now state an important tool for working with positive harmonic functions, the \emph{Harnack inequalities} for the ball. These say that if $v$ is a positive harmonic function on the ball $B(0,2\eps)$
and if $z\in B(0,\eps) \subseteq B(0,2\eps)$ then 
\begin{equation}\label{eq:harnack-ball}
\frac{1}{3} \leq \frac{v(z)}{v(0)} \leq 3\,.
\end{equation} A statement and proof of this result can be found in \cite[pg.~47]{harmonic-axler}.  For our purposes, we need a slight generalization:

\begin{lemma} \label{lem:harnack-general}
	Let $v$ be a positive harmonic function on an open set $\Omega \subset \C$. For $\eps > 0$, let $z_1, z_2 \in \Omega$ be points at distance $d := d(z_1,z_2)$ so that all $z_3$ that lie on the line segment joining $z_1$ and $z_2$ have $B(z_3,\eps) \subset \Omega$. Then
	\begin{equation}
	\frac{1}{3^{2d/\eps + 1}} \leq \frac{v(z_1)}{v(z_2)} \leq 3^{2d/\eps + 1}\,.
	\end{equation}
\end{lemma}

\vspace{4mm}

We will also need the following lemma, which is a simple consequence of the Harnack inequalities. 

\begin{lemma} \label{lem:ratioOfPoisson} For $\eps >0$, let $\{P_z\}$ be the Poisson density of the ball $B(1,2\eps)$. Then for all $z \in B(1,\eps)$ we have
\begin{equation} \label{equ:ratioOfPoisson} \max_{s \in \p B_{2\eps}} \frac{P_{z}(s)}{P_{1}(s)} \leq 3.\end{equation}
\end{lemma} 

\vspace{4mm}

Proofs of both Lemma~\ref{lem:harnack-general} and Lemma~\ref{lem:ratioOfPoisson} are basic and we postpone their short proofs to Appendix~\ref{app:proofs}.

We now prove a straight-forward lemma that will allow us to find a large negative value of $u_0$. For this lemma, we note that if $\{P_z\}$ is the Poisson density of 
$B(1,\eps)$ then, by symmetry, $P_1(s) = \frac{1}{2\pi\eps}$ for all $s \in \partial B(1,\eps)$.

\begin{lemma}\label{lem:neg-boundary} Let $\eps >0$ and 
let $v(z)$ be a symmetric, harmonic function on $B(1,\eps)$ with $v(1) = 0$, let $\{P_{z}\}$ be the Poisson density of the ball $B(1,\eps)$ and set
$$ M := \int_{s \in \partial B(1,\eps)} |v(s)|P_1(s) \,ds.$$ 
Then there is a $z_0 \in \partial B(1,\eps)$ with $\im(z_0) \geq 0$ and $ v(z_0) \leq - M/2$. 
\end{lemma}
\begin{proof}
Let us set $B_\eps := B(1,\eps)$, $P(s) := P_{1}(s)$ and define $\p B_\eps = E_+ \cup E_-$ where $E_+ = \{ s \in \p B_\eps : v(s)\geq 0 \}$ and $E_+ = \{ s \in \p B_\eps: v(s)\leq 0 \}$. Now put 
\[ A = \int_{s \in E^+} |v(s)|P(s)\,ds \quad\text{and}\quad  B = \int_{s \in E^-} |v(s)|P(s)\, ds. 
\] and note that 
\[ 0 = v(1) = \int_{s \in \partial B_\eps} v(s) P(s)\,ds = \int_{s \in E^+} |v(s)|P(s)\,ds  - \int_{s \in E^-} |v(s)|P(s)\,ds  = A-B,
\] and that $A+B = M$. Thus 
\[ M/2 = \int_{E^-} |v(s)|P(s) \, ds \leq \max_{s\in E^{-}} |v(s)| \int_{E^-} P(s) \, ds \leq \max_{s\in E^{-}} |v(s)|, \]
where the last line holds due to the fact that $P(s) = 1/(2\pi\eps)$ and the length of $E^-$ is at most $2\pi\eps$. So if we let $z_0$ be a value which attains 
this maximum, we note that both $v(z_0) = v(\bar{z}_0) \leq -M/2$ by symmetry. Hence one of $z_0,\bar{z}_0$ will have non-negative imaginary part, as desired. 
\end{proof}

\subsection{Proof of Lemma~\ref{lem:money} } 
We now are in a position to prove the main result of this section, Lemma~\ref{lem:money}.

\vspace{4mm}

\noindent \emph{Proof of Lemma~\ref{lem:money}. }
To reduce clutter, let us define $B_{\eps}, B_{2\eps}$ to be $B(1,\eps),B(1,2\eps)$, respectively. Let $P_z$ be the Poisson density of $B_{2\eps}$ and put
\[ M := \int_{s \in \partial B_{2\eps}} |u_0(s)|P_{1}(s) \,ds .
\]By the definition of the Poisson density, for each $z \in B_{\eps}$, we have 
\[ u_0(z) =  \int_{s \in \p B_{2\eps} } u_0(s)  P_{z}(s)\, ds 
\] 
and since, for all $z \in B_{\eps}$, we have $\max_{s \in \p B_{2\eps}} \, P_{z}(s)/P_{1}(s) \leq 3 $ (by Lemma~\ref{lem:ratioOfPoisson})
we obtain  \begin{equation} \label{eq:u-UB}
 \max_{z \in B_{\eps}} |u_0(z)| \leq  3 \int_{s\in \p B_{2\eps}} |u_0(z)|P_{1}(s) \, ds = 3M. 
\end{equation}
We now apply Lemma~\ref{lem:neg-boundary} to the function $v := u_0(z)$ (which is harmonic, symmetric and has $u_0(1) = 0$) and the region $B_{2\eps}$ to find a point $z_0 \in \p B_{2\eps}$ for which $u_0(z_0) \leq  -M/2$ and $\im(z_0) \geq 0$.  By the planar geometry calculation Lemma \ref{lem:planar-geometry1}, we may write $z_0$ in the form $z_0 = \rho e^{i\phi}$ where $\phi \in [0,4\eps]$ and $\rho \in [1-2\eps,1+2\eps]$.
 
Another planar geometry calculation (Lemma \ref{lem:planar-geometry2}) ensures that $\rho e^{4\eps i} \in B(1,8\eps)$, allowing us to make use of the $b$-decreasing hypothesis: 
\begin{align*}
 u_0(\rho e^{4\eps i}) &\leq u_0(z_0) +b \\
 &\leq b-M/2\,.
 \end{align*} 
 We now apply Lemma~\ref{lem:u0-positive} to see that $u_0(\rho) \geq 0$ as $\rho >0$, thereby allowing us to obtain a bound on $h_{4\eps,b}(\rho)$
\begin{equation} \label{equ:inmoney1}  h_{4\eps,b}(\rho) = u_0(\rho) - u_0(\rho e^{i4\eps})+b \geq -u_0(\rho e^{i4\eps}) + b \geq M/2. 
\end{equation} 
We know that $S(-4\eps,4\eps) \cap B_{4\eps} \subset B_{8 \eps}$ (Lemma~\ref{lem:planar-geometry2}) and $u$ is $b$-decreasing in $B_{8\eps}$.  We then know that $h_{4\eps,b}(z)$ is harmonic and positive (Lemmas~\ref{lem:h-phi-HarmInSector} and \ref{lem:h-phi-+veInSector}) in $S(-2\eps,2\eps) \cap B_{4\eps}$ and thus  
we may apply Lemma~\ref{lem:harnack-general} to learn that 
\begin{equation} \label{equ:inmoney2}  h_{4\eps,b}(\rho) \leq 3^{64\eps/\eta + 1} h_{4\eps,b}(\rho e^{i(2\eps -\eta/2)} ),
\end{equation} 
since the distance $d(\rho,\rho e^{i (2\eps - \eta/2)}) \leq (1+2\eps)2\eps \leq 4\eps$ and each point on the segment between them is at least $\rho\cdot\eta /4 \geq \eta/8$ from the boundary of $S(-2\eps,2\eps) \cap B_{4\eps}$.  

Now observe that at the value of $z = \rho e^{i(2\eps - \eta/2)}$, we have $h_{4\eps,b}(z) = \varphi_{\eta,b}(z)$. That is,

\begin{equation} \label{equ:inmoney3}
 h_{4\eps,b}(\rho e^{i(2\eps - \eta/2)}) = u(\rho e^{i(2\eps - \eta/2)}) - u(\rho e^{i(2\eps + \eta/2)}) + b = \vp_{\eta,b} (\rho e^{i(2\eps-\eta/2)} ). \end{equation} We now apply Lemmas~\ref{lem:h-phi-HarmInSector} and \ref{lem:h-phi-+veInSector} 
 to learn that $\vp_{\eta,b}$ is harmonic and positive in 
\[ S(-\eta/2, 4\eps - \eta) \cap B_{4\eps}. 
 \] Hence we may apply Lemma~\ref{lem:harnack-general} along with the fact that $\rho \in [1- 2\eps,1+2\eps]$ to see that 

\begin{equation} \label{equ:inmoney4} \vp_{\eta,b} (\rho e^{i(2\eps-\eta)}) \leq 3^{32 \eps/\eta + 1}   \vp_{\eta,b} (1) \,, 
\end{equation}  
since $|\rho e^{i(2\eps - \eta/2)}-1| \leq 4\eps$ and each point on the segment between them is at least $\eta/4$ from the boundary of 
$S(-\eta/2, 4\eps - \eta) \cap B_{4\eps}$.
Thus, chaining together lines lines~\eqref{eq:u-UB},\eqref{equ:inmoney1},\eqref{equ:inmoney2},\eqref{equ:inmoney3},\eqref{equ:inmoney4} gives
\[ \max_{z \in B_{\eps}} |u_0(z)| \leq  3^4 \cdot 3^{96 \eps/ \eta}\vp_{\eta,b}(1) ,
\] as desired. 
\qed 

\section{Bounding the tail of the cumulant sequence} \label{sec:TailOfcumulants}

In the previous section we showed how to control the maximum of $u_0$ in a small ball around $1$, in terms of the much more amenable function $\vp_{\eta,b}$.
In this section, we use this bound to prove that the 
normalized cumulant sequence $(a_j)_{j\geq 2}$ has nice decay properties.

\begin{lemma} \label{lem:CumulantDecay}  For $\eps \in (0,2^{-4})$ and $b \geq 0$, let $u$ be a
$b$-decreasing function, weakly-positive, symmetric and harmonic function in $B(1,2^{4}\eps)$.
Let $(a_j)_{j\geq 1}$ be the normalized cumulant sequence of $u$. If
\begin{equation} \label{equ:condOnb}
 \sum_{j\geq 2} |a_j|\eps^j > b, 
\end{equation} 
then for all $L \geq 2$ we have
\begin{equation} \label{equ:CumulantFrac} \frac{ \sum_{j\geq L } |a_j|\eps^j }{ \sum_{j\geq 2} |a_j|\eps^j } \leq C \cdot 2^{-L}, \end{equation}
where $C >0$ is an absolute constant\footnote{Indeed, we can take $C = 3^{390}$}.
\end{lemma}

\begin{proof} We first point out that the expression in \eqref{equ:CumulantFrac} makes sense; the denominator is non-zero from the strict inequality at \eqref{equ:condOnb} and the numerator is finite since we may write 
\begin{equation}\label{eq:U0-in-cumulant-decay} U_0(w) = \sum_{j\geq 2} a_j\Re(w^j),
\end{equation} by  Lemma~\ref{lem:USeries}, for all $ w \in B(0,8\eps)$ and thus the series $\sum a_j \eps^j$ is absolutely convergent.

To prove Lemma~\ref{lem:CumulantDecay}, the idea is to compare both the numerator and denominator in \eqref{equ:CumulantFrac} to $\vp_{\eps,b}(1)$.
We begin with the denominator. Recalling that \[ \vp_{\eps,b}(e^{w}) = u(e^{w}) - u(e^{w+i\eps}) + b,\] 
we use \eqref{eq:U0-in-cumulant-decay} to express
\[ \vp_{\eps,b}(e^{w}) = \sum_{j\geq 2} a_j \Re( w^j - (w+i\eps)^j) + b, 
\] for $w$ sufficiently small. And so, setting $w = 0 $, we obtain  
\begin{equation} \label{equ:lowerBoundOnDenom}
 |\vp_{\eps,b}(1)| \leq  \sum_{j\geq 2} |a_j| \eps^j + b \leq 2\sum_{j\geq 2} |a_j| \eps^j, \end{equation} by the triangle inequality
 and our assumption at \eqref{equ:condOnb}.

 We now turn to obtain an upper bound on the numerator of (\ref{equ:CumulantFrac}).  We apply Cauchy-Schwarz to obtain
\begin{equation} \label{equ:cauchySchwartz} \sum_{j\geq L } |a_j|\eps^j \leq  2\left( \sum_{j\geq 2 } |a_j|^2(2\eps)^{2j} \right)^{1/2} 2^{-L}. 
\end{equation} 
We now look to relate the series on the right-hand-side of \eqref{equ:cauchySchwartz} to $U_0$.
In preparation for this, we write
\[ U_0(\rho,\t) := U_0(\rho e^{i\t}) =  \sum_{j\geq 2} a_j\rho^{j}\cos(j\t),
\] which is valid for all $|\rho| < 8\eps$, due to \eqref{eq:U0-in-cumulant-decay}, and then use Parseval's Theorem to write  
\begin{equation} \label{equ:parseval}
  \sum_{j\geq 2 } |a_j|^2(2\eps)^{2j}  =  \frac{1}{2\pi}\int_0^{2\pi} |U_0(2\eps,\t)|^2 d\t \leq \max_{\t \in [0,2\pi]} |U_0(2\eps,\t)|^2. 
\end{equation}
As a side remark, note that \eqref{equ:parseval}, along with the strict inequality at \eqref{equ:condOnb}, implies that $\max_{\t \in [0,2\pi]} |U_0(2\eps,\t)| >0$.
 
Returning to the main arc of the proof, recall that $z = e^{w}$ and $w = \rho e^{i\theta}$; so as $\t$ ranges over $[0,2\pi]$, $z$ lies on the curve
\[ \Gamma = \{ \exp(2\eps e^{i\theta}) : \t \in [0,2\pi] \},
\] which is contained in the ball $B(1,4\eps)$, due to the inequality $|1-\exp((2\eps)e^{i\t})| \leq 4\eps$, which holds for $\eps <1 $. 
Hence we may bound the right hand side of \eqref{equ:parseval} 
\begin{equation} \label{equ:boundInball}
 \max_{\t \in [0,2\pi]} |U_0(2\eps,\t)|^2 \leq \max_{z \in B(1,4\eps)} |u_0(z)|^2. 
\end{equation}
Here is the key ingredient: we apply Lemma~\ref{lem:money} to obtain an upper bound on $u_0$ in terms of $\vp_{\eta,b}$ in $B(1,4\eps)$ with $\eta = \eps$:
\begin{equation} \label{equ:u0andphi} 
\max_{z \in B(1,4\eps)} |u_0(z)|^2  \leq  \left( 3^{388} \vp_{\eps,b}(1) \right)^2. 
\end{equation}
Note that this also implies that $\vp_{\eps,b}(1) >0$, due to the remark after \eqref{equ:parseval} and \eqref{equ:boundInball},\eqref{equ:u0andphi}.

To finish, we put together the lower bound at (\ref{equ:lowerBoundOnDenom}) on the denominator in (\ref{equ:CumulantFrac}) with the 
upper bound on the numerator, coming from \eqref{equ:cauchySchwartz}, \eqref{equ:parseval}, \eqref{equ:boundInball} and \eqref{equ:u0andphi}, to obtain 
\[  \frac{ \sum_{j\geq L } |a_j|\eps^j }{ \sum_{j\geq 2} |a_j|\eps^i } \leq  2 \cdot \frac{ 3^{388}\vp_{\eps,b}(1)}{\frac{1}{2}\vp_{\eps,b}(1)}2^{-L} \leq 3^{390} 2^{-L}, 
\] as desired.
\end{proof}

\section{Taming the cumulant sequence} \label{sec:BES}

In this section we provide a third and final ingredient in our proof of Lemma~\ref{thm:maintechnical}. In Section~\ref{sec:TailOfcumulants}, we 
showed that the sequence $(a_j)_{j\geq 2}$ had to have quite a 
bit of its ``mass'' concentrated on the early terms. In this section, we use our weak-positivity hypothesis to show that, in this situation, we can control all of the cumulants in terms of the \emph{second} cumulant, the variance. 

The main result of this section is Lemma~\ref{lem:BES}, which can be seen as a quantitative version of a tool co-discovered by De Angelis \cite{angelis} 
and Bergweiler, Eremenko and Sokal \cite{BES} which was used in their work on classifying polynomials whose large powers have all positive coefficients. It is also a relative of Lemma 7 in the previous work of the authors \cite{clt1} and can be viewed as an effective form of Marcinkiewicz's Theorem \cite{Marcin}.

To prove Lemma~\ref{lem:BES} we need the following preparatory lemma, which is an elementary fact about sequences of non-negative real numbers.

\begin{lemma} \label{lem:bigTermBeforeCutoff}
Let $A \geq 1$, $s >0$ and let $(c_i)_{i\geq 1}$ be a sequence of non-negative numbers for which
the sum $\sum_{i\geq 1} c_is^i $ converges and is non-zero. If $L \in \N$ is such that 
\begin{equation} \label{eq:truncation}
\sum_{i=1}^{L} c_is^i >  \sum_{i > L} c_is^i
\end{equation} then there exists an $\ell \in \{1,\ldots,L\}$ and $s_{\ast} > s(16A)^{-(L+1)}$ so that 
\begin{equation} \label{equ:finalCond} c_{\ell}s_{\ast}^{\ell} > A \sum_{i \neq \ell} c_is^i_{\ast}. \end{equation} 
\end{lemma}

\begin{proof}
	To start, we choose $s_0 := s/(2A)$. This immediately gives us
\begin{equation} \label{eq:truncation-A}
\sum_{i=1}^{L} c_is_0^i >  (2A) \sum_{i > L} c_is_0^i. 
\end{equation} 

	We now define an algorithm that will find $\ell \in \{1,\ldots,L\}$ and $s_{\ast} > s(16A)^{-(L+1)}$ that satisfies \eqref{equ:finalCond}: Initialize $t = 0, s_0 = s_0$ (defined above) and $j_0 = L$ and inductively define a sequence of integers $j_1 \geq j_2 \geq \cdots \geq 1$ and positive real numbers $s_1 > s_2 > \cdots$ as follows: if the pair $(j_t,s_t)$ satisfies 
 \begin{equation}\label{equ:HaltCond}
	c_{j_t} s_t^{j_t} > 2A \sum_{\substack{1 \leq i \leq L \\ i \neq j_t}} c_i s_t^i\,. 
	\end{equation}
then we terminate and return $(\ell,s_\ast) =(j_t,s_t)$.  Otherwise, choose $j_{t+1}$ so that 
	\begin{equation} \label{eq:j+1-def}
	c_{j_{t+1}} s_t^{j_{t+1}} =  \max\{ c_1 s_t^1, c_2 s_t^2, \ldots , c_ts_t^{j_t} \}
	\end{equation}
	and set $s_{t + 1} = s_t /(16A)$. To see that this algorithm successfully produces a pair $(\ell,s_\ast)$ that satisfies the conclusions of the lemma, we prove two claims.
\begin{claim}	
For each $t\geq 0$ we have 	
	\begin{equation} \label{equ:condInIteration} 
	c_{j_t}s_t^{j_t} >  4A\sum_{  i = j_t+1}^L c_is_{t}^i\,.
	\end{equation} \end{claim}
	
\noindent \emph{Proof of Claim.}
	We apply induction on $t$; note that the $t = 0$ case is trivial.  Now suppose  \eqref{equ:condInIteration} is satisfied for some $t \geq 0$, 
	write $a = j_{t}, b = j_{t+1}$ (for ease of notation) and recall that $b = j_{t+1}$ was chosen so that 
	$c_{b}s_t^{b} = \max_{1 \leq i \leq a}\{ c_i s_t^i \}$; thus,  
	\begin{equation*} 
	c_b s_t^{b}  > \sum_{ i = b+1 }^{a} \frac{c_is_t^i}{2^{i-b}}\,, 
	\end{equation*} 
	and since $s_{t+1} = s_t/(16A)$, we have 	
	\begin{equation} \label{eq:bes1} c_b s_{t+1}^{b} = (16A)^{-b} c_b s_t^b  > (16A)^{-b}  \sum_{ i = b+1 }^{a} \frac{c_is_t^i}{2^{i-b}}  \geq 8A \sum_{i=b+1}^a c_is_{t+1}^i. \end{equation}
	 By the induction hypothesis at \eqref{equ:condInIteration}, we have $c_{a}s_t^{a} \geq 4A\sum_{i = a+1}^{ L} c_is_t^i \,$ and thus (crudely) we have
	\begin{equation} \label{eq:bes2} c_{b}s_{t+1}^{b} \geq  8A\sum_{i = a+1}^{ L} c_is_{t+1}^i. \end{equation}
	Averaging \eqref{eq:bes1} and \eqref{eq:bes2} yields
	\begin{equation*}
	c_{b}s_{t+1}^{b}  > 4A \sum_{i= b+1}^L c_is_{t+1}^i,
	\end{equation*} as desired.
	This completes the proof of the claim, by induction. \qed
\begin{claim} \label{claim:bes2}
We have that $j_1 > j_2 > \cdots \geq 1$ is a strictly decreasing sequence of integers.
\end{claim}

\noindent \emph{Proof of Claim.} By definition, we have $j_1 \geq j_2 \geq \cdots$ and so we claim that if $j_{t+1} = j_t$, then the pair $(j_t,s_t)$ would in fact satisfy \eqref{equ:HaltCond}, the halting condition for the algorithm. So suppose that $j := j_{t+1} = j_t$ and recall that $s_{t+1} = s_t/(16A)  $; 
then by \eqref{eq:j+1-def}, we have, for all $i \leq j$, 
\[ c_is_t^i \leq c_js_t^j(16A)^{i-j}. \]
	This implies 
	\begin{equation}\label{eq:bes3} 4A\sum_{i = 1}^{j-1} c_i s_{t}^i \leq c_js_t^j\left( 4A \sum_{i = 1}^{j-1} (16A)^{i-j} \right) \leq c_j s_{t}^j\,.
	\end{equation} Averaging \eqref{eq:bes3} and \eqref{equ:condInIteration} yields \eqref{equ:HaltCond} for $(\ell,s_\ast) = (j_{t},s_t)$, implying that the algorithm would have halted before proceeding to step $t+1$, a contradiction. \qed
	
\vspace{4mm}	
	
	 Thus, Claim~\ref{claim:bes2} tells us that the algorithm must terminate in at most $L$ steps and thus $s_{\ast} > s(16A)^{-(L+1)}$.
	
	To see that we have found a pair $(\ell,s_{\ast})$ that also satisfies \eqref{equ:finalCond}, we simply note that \eqref{eq:truncation-A} implies
	$\sum_{i=1}^{L} c_is_{\ast}^i >  (2A) \sum_{i > L} c_is_{\ast}^i$ and thus, averaging this with \eqref{equ:HaltCond}, yields the inequality \eqref{equ:finalCond},
	as desired.\end{proof}

\vspace{4mm}

For our main lemma of this section we make use of the (somewhat crude) inequalities, which are checked in Appendix~\ref{app:proofs}.
\begin{fact} \label{fact:negcos}
For $j \geq 3$, we have 
\begin{equation} \label{equ:Factdiff-ve} \min_{\t \in \R}\{ ( \cos \t )^j - \cos j\t \} < -1/2 ;
\end{equation}
\begin{equation} \label{equ:Factdiff+ve}  \max_{\t \in \R}\{ ( \cos \t )^j - \cos j\t \} > 1/2. \end{equation}
\end{fact}

\vspace{4mm}

As mentioned before, we apply a clever idea from the work of De Angelis and Bergweiler, Eremenko and Sokal and use the non-negativity of another ``difference function'': \[u(|z|) - u(z).\]

\begin{lemma} \label{lem:BES} For $s \in (0,1/2) $ and $L\geq 2$, let $u$ be a weakly-positive, symmetric harmonic function on $B(1,2s)$ and let $(a_j)_{j}$ be its normalized cumulant sequence. If $(a_j)_{j\geq 2}$ is a non-zero sequence and 
	\begin{equation} \label{equ:cutoff} \sum_{j\geq 2}^{L} |a_j|s^j \geq \sum_{j> L} |a_j|s^j, \end{equation} then there exists a real number $s_{\ast} > s2^{-6(L+1)}$ for which 
	$ |a_2| \geq s^{j-2}_{\ast}|a_j|, $ for all $j \geq 2$.\end{lemma}
\begin{proof}
	First note that the function $U(w) = u(e^w)$ is harmonic for $w \in B(0,s)$ due to the inequality $|e^w - 1| \leq 2|w|$ for $|w| \leq 1/2$ and the fact that $u$ is harmonic on $B(1,2s)$. We consider the function
	\begin{equation} \label{equ:+veDiff} U_0(\Re(w)) - U_0(w) = u_0(|e^{w}|) - u_0(e^w) \geq 0, 
	\end{equation} where the inequality follows from weak positivity. Now, writing $w = \rho e^{i\t}$ and considering the series expansion of $U_0$ around $w=0$ (Lemma~\ref{lem:USeries}), we have 
	\[ F(\rho,\t) := U_0(\rho \cos \t ) - U_0( \rho e^{i\t} ) = \sum_{j \geq 2} a_j\rho^j ( (\cos \t)^j - \cos j\t ),
	\] for all $0 \leq \rho < s $.
	Since $a_j$ is not identically $0$ for all $j\geq 2$, we may apply Lemma~\ref{lem:bigTermBeforeCutoff} to the sequence $(|a_j|)_{j\geq 2}$ with $A = 4$, to get an integer $\ell \in [L]$ and a real number $s_{\ast} > s2^{-6(L+1)}$ so that
	\begin{equation} \label{equ:starBiggest} |a_{\ell}|s_{\ast}^{\ell} > 4 \sum_{2\leq i \not= \ell} |a_i|s_{\ast}^i.
	\end{equation}
	We now use weak-positivity to see that $\ell = 2$. For this, assume $\ell > 2$ and apply Fact~\ref{fact:negcos} to find a $\t_0$ for which 
	\begin{equation} \label{equ:j0term}  a_{\ell}\big(( \cos \t_0 )^{\ell} - \cos \ell\t_0 \big) \leq -|a_{\ell}|/2. \end{equation} We write 
	\[ F(s_{\ast},\t_0) = a_{\ell} s_{\ast}^{\ell}\big(( \cos \t_0 )^{\ell} - \cos \ell\t_0 \big) + \sum_{2 \leq j \not= \ell} a_{j}s_{\ast}^{j}\big(( \cos \t_0 )^{j} - \cos j \t_0 \big)
	\] and apply (\ref{equ:j0term}) to bound the first term on the right-hand-side and apply the triangle-inequality to bound the sum. We obtain
	\[ F(s_0,\t_0) \leq \frac{-|a_{\ell}|s_{\ast}^{\ell}}{2} + 2 \sum_{2 \leq j \not= \ell } |a_j|s_{\ast}^j < 0,
	\] where the last inequality follows from (\ref{equ:starBiggest}). However this contradicts the positivity of $F$ (\ref{equ:+veDiff}). We therefore conclude that $\ell =2$ and so, from (\ref{equ:starBiggest}) again, we have that
	\[  |a_{2}|s_{\ast}^{2} > 4 \sum_{i \geq 3} |a_i|s^i_{\ast} \geq 4|a_j|s^j_{\ast},
	\] for any $j \geq 3$, as desired.  
\end{proof}

\section{Proof of Lemma~\ref{thm:maintechnical} : the final stroke} \label{sec:technical-proof}

In this section we combine the ingredients from the previous sections to prove Lemma~\ref{thm:maintechnical}. What we state here is slightly stronger than what we need
but we make these bounds explicit for use in later work.

\begin{lemma} \label{thm:maintechnical}
	For $\eps \in (0,1)$, $b \geq 0$ and $n \geq 1$, let $X \in \{0,\ldots,n\}$ be a random variable with standard deviation $\s >0$, logarithmic potential $u = u_X$ and  normalized cumulant sequence $(a_j)_{j\geq 1}$.  If $u$ is $b$-decreasing and harmonic in $B(1,\eps)$ and 
\begin{equation}\label{eq:not-point-mass} \sum_{j\geq 2} |a_j|\left(\eps/32\right)^j > b, \end{equation} then  $\psi_{X^{\ast}}$, the characteristic function of $X^{\ast} := (X-\mu)\s^{-1}$, satisfies
\[ \psi_{X^{\ast}}(\xi) = \exp( -\xi^2/2 + R(\xi) ), 
\] where 
\begin{enumerate}
\item $R(0) = R^{(1)}(0) = R^{(2)}(0) = 0$ and 
\[ |R^{(\ell)}(0)| \leq \ell!(c_2\s)^{2-\ell},\] for all $\ell \geq 3$.
\item In particular\footnote{We can take the constants $c_1 = 2^{3246} $, $c_2 = 2^{-3246}$ .},
\[ |R(\xi)|\leq   \frac{ c_1|\xi|^3}{\eps \s }  \]
for all $\xi \in \C$ with $|\xi| \leq  c_2\eps \sigma$.  
\end{enumerate}
\end{lemma}
\begin{proof}
	Let $\psi_X(\xi) = \EE_X e^{i \xi X }$ be the characteristic function of $X$, and note that
	\[ \psi_X(\xi) = \exp\left( \sum_{j\geq 1} \frac{\kappa_j}{j!} (i\xi)^j \right) = \exp\left( \sum_{j\geq 1} a_j (i\xi)^j \right), 
	\] where $\k_j$ is the $j$th cumulant of $X$ and $(a_j)_{j}$ is the normalized cumulant sequence. Here, this expansion is valid for all $|\xi| <  \eps/2$ since harmonicity of $u$ in $B(1,\eps)$ implies analyticity of $\psi$ in $B(0,\eps/2)$ due to the inequality $|1 - e^w| \leq 2|w|$ for $|w| < 1/2$.
	Now note that $\psi_{X^{\ast}}(\xi) = \psi_X(\xi/\s)e^{-i\frac{\mu \xi}{\s} }$ is the characteristic function of $X^{\ast}$. 
	Using the fact that $a_1 = \mu$ and $a_2 = -\s/2$, as noted at \eqref{eq:k1-is-mean} and \eqref{eq:k2-is-var}, we have 
	\begin{equation} \label{equ:CharFunctiofZ} 
	\psi_{X^\ast}(\xi) = \exp\left( -\xi^2/2 + \sum_{j\geq 3} \frac{a_j}{\s^j} (i\xi)^j \right)
	\end{equation} and so we define 
	\begin{equation}\label{eq:R-expansion} R(\xi) :=  \sum_{j\geq 3} \frac{a_j}{\s^j} (i\xi)^j.
	\end{equation} We now apply Lemma~\ref{lem:CumulantDecay} to bound $R$. To see that we may apply this lemma, note that
	\eqref{eq:not-point-mass} implies condition \eqref{equ:condOnb} in Lemma~\ref{lem:CumulantDecay}; the logarithmic potential 
	$u = u_X$ is weakly-positive and symmetric in $B(1,\eps)$ (as noted in Section~\ref{sec:defs});
	and $u$ is $b$-decreasing and harmonic in $B(1,\eps)$, by assumption. Therefore
	\[ \frac{\sum_{j\geq L} |a_j|(\eps/32)^j }{\sum_{j\geq 2} |a_j|(\eps/32)^j} \leq C\cdot 2^{-L}, 
	\] for all $L\geq 2$. Now, if we choose $L = 2 +  \log_2 C$ we have that
	
	\begin{equation} \label{equ:boundOnCumulants} 
	\sum_{j=2}^{L} |a_j| (\eps/32)^j > \sum_{j > L} |a_j| (\eps/32)^j \,,   
	\end{equation}
	and so we may apply Lemma~\ref{lem:BES} with $L = 539$ and $s = \eps/32 $ 
	to obtain a $s_{\ast} > 2^{-3245} \eps$ for which 
	\begin{equation} \label{eq:tamed-cumulants} \s^2 = |a_2| > s^{j-2}_{\ast}|a_j|. \end{equation} 
	And so for $j \geq 3$ the $j$th term in the expansion of $R(\xi)$ is 
	\[ \frac{|R^{(j)}(0)|}{j!} = |a_j|\s^{-j} \leq (s_{\ast}\s)^{2-j}, 
	\]and so, for $|\xi| < s_{\ast}\s$, we have
	\[ |R(\xi)| \leq \sum_{j \geq 3} \frac{|a_j||\xi|^j}{\s^j} \leq  \sum_{j\geq 3} \frac{|\xi|^j}{(s_{\ast} \s )^{j-2}}
	 = \frac{|\xi|^3}{s_{\ast} \s(1 - |\xi|/(s_{\ast}\s))}\,. \] 
	This means that we can factor 
	\[ \psi_{X^\ast}(\xi) = e^{-\xi^2/2}e^{R(\xi)}, \]
	where $|R(\xi)| < 2\frac{|\xi|^3}{s_{\ast} \s } \leq \frac{2^{3246}}{\eps \sigma},$ for $|\xi| < (s_{\ast}\s)/2 \leq 2^{-3246} \eps \sigma$. This completes the proof of Lemma~\ref{thm:maintechnical}.
\end{proof}

\section{Proofs of Theorems~\ref{thm:Polysector} and \ref{thm:logn}} \label{sec:proofs-of-thms}

In this section we use Lemma~\ref{lem:b-decreasing} along with our main technical lemma, Lemma~\ref{thm:maintechnical} to deduce our theorems on univariate polynomials. Before we finish these proofs, we need to quickly derive our ``Fourier-inversion'' lemma,
which allows us to conclude that $X$ is approximately normal based on the hypothesis that the characteristic function $\psi_X$ is approximately 
the characteristic function of a normal.

\subsection{Fourier inversion}
In this short subsection, we derive the following basic ``Fourier inversion'' tool.

\begin{lemma}\label{lem:fourier-inversion} Let $X \in \R$ be a random variable with characteristic function $\psi$. If 
\[ \psi(\xi) = \exp\left( -\xi^2 /2+ R(\xi) \right),
\] where $|R(\xi)| \leq \eta|\xi|^3$ for all $|\xi| < \tau$ then 

\[ \sup_{t \in \R} \left| \PP( X \leq t ) - \PP(Z \leq t) \right| \leq 2^9\max\{\eta,\tau^{-1}\} ,
\] where $Z \sim N(0,1)$.\end{lemma}

\vspace{4mm}

Following Lebowitz, Pittel, Ruelle, and Speer, we use the following quantitative result which can found in the textbook of Feller \cite[pg.~538]{feller}. 

\begin{lemma} \label{lem:Inversion} Let $Z \sim N(0,1)$ be a standard normal, let $X \in \R$ be a random variable and let $\psi(\xi)$ be its
	characteristic function. Then, for all $T >0$, we have 
	\begin{equation} \label{equ:Inverson} \sup_{t \in R}| \PP( X \leq t)  - \PP(Z \leq t) | \leq \frac{1}{\pi} \int_{-T}^T \left| \frac{ \psi(\xi) - e^{-\xi^2/2} }{\xi} \right| \, d\xi + \frac{4}{T}. \end{equation} 
\end{lemma}

\vspace{4mm}

We may now easily derive Lemma~\ref{lem:fourier-inversion}, our Fourier inversion lemma.

\vspace{4mm}

\noindent\emph{Proof of Lemma~\ref{lem:fourier-inversion}:} 
We apply Lemma~\ref{lem:Inversion} with $T = \min\{ \eta^{-1}/8,\tau \}$ to obtain
\begin{equation} \label{eq:smoothing}
\sup_{t \in \R}| \PP(X \leq t) - \PP(Z \leq t) | \leq \frac{1}{\pi} \int_{-T}^T \left| \frac{ \psi(\xi) - e^{-\xi^2/2} }{\xi} \right| \, d\xi + \frac{4}{T}.
\end{equation}
First note that we may assume $\eta < 1$, otherwise the theorem is trivial.
Let $I := \int_{-T}^T \left| \frac{ \psi(\xi) - e^{-\xi^2/2} }{\xi} \right|\,d\xi $ be the integral in \eqref{eq:smoothing} and set $a := \eta^{-1/3}$. 
We bound the $I$ by breaking the integral into two ranges: $|\xi| \in [0,a]$ and $|\xi| \in [a,T]$. For $|\xi| \leq a$, we bound the integrand \begin{equation} \label{eq:smooth-low-range}
\left|\frac{\psi(\xi) - e^{-\xi^2 /2}}{\xi} \right| = e^{-\xi^2/2}\left|\frac{e^{R(\xi)} - 1}{\xi}\right| 
\leq 4\eta e^{-\xi^2 / 2} |\xi|^2, \end{equation} since $|e^z - 1| \leq 4|z|$ for $|z| \leq 1$. For $a \leq |\xi| \leq  T$, we use the fact that $a = \eta^{-1/3} \geq 1$ to bound the integrand 
\[ \left|\frac{\psi(\xi) - e^{-\xi^2 /2}}{\xi} \right|  \leq 2e^{-\xi^2/4} \cdot e^{-\xi^2/4 + |R(\xi)|}\, , \]
where 
\[ -\xi^2/4 + |R(\xi)| = |\xi|^2\left(-\frac{1}{4} + \eta|\xi|\right) \leq - \xi^2 / 8 \, \]
and the last line holds due to the fact that $\eta|\xi| \leq \eta T \leq 1/8$, by the choice of $T$.
So, for $|\xi| \in [a,T]$, we have 
\begin{equation} \label{eq:smooth-up-range} \left|\frac{\psi(\xi) - e^{-\xi^2 /2}}{\xi} \right|\leq 2 e^{-\xi^2 / 4} e^{-\xi^2 / 8} \leq 2 e^{-\xi^2 / 4} e^{- a^2/8 } \leq  32\eta e^{-\xi^2 / 4}, \end{equation}
due to the fact that 
$\exp(- |x|^{2/3} / 8) \leq 16/|x|$.  Using \eqref{eq:smooth-low-range} and \eqref{eq:smooth-up-range} we can bound $I$ by 
\begin{equation} I \leq  8\eta \int_0^a e^{-\xi^2 / 2} |\xi|^2 \, d\xi + 64\eta \int_a^T e^{-\xi^2 / 4} \, d\xi \leq 2^9\eta , \end{equation}
where we have used the facts 
\[\int_{-\infty}^{\infty} e^{-t^2/2}|t|^2\, dt = \sqrt{2\pi}\, \, \,  \textit{    and     }\, \, \,  \int_{-\infty}^{\infty} e^{-t^2/4}\, dt = 2\sqrt{2\pi} .\]

Putting this together with \eqref{eq:smoothing} gives the bound 
$$\sup_{t \in \R} |\PP(X \leq t) - \PP(Z \leq t)| \leq  \pi^{-1}I + 32 \max\{\eta,\tau^{-1}\} \leq 2^9 \max\{\eta,\tau^{-1}\},$$
as desired.\qed

\vspace{4mm}

We note that Lemma~\ref{lem:fourier-inversion}, along with our main technical Lemma~\ref{thm:maintechnical}, implies the following 
general result for random variables $X \in \{0,\ldots,n\}$.

\begin{corollary}\label{cor:CLT}
For $\eps \in (0,1)$ and $b \geq 0$, let $X \in \{0,\ldots,n\}$ be a random variable with logarithmic potential $u_X$ and normalized cumulant sequence $(a_j)_{j}$. If $u_X$ is $b$-decreasing and harmonic in $B(1,\eps)$ and 
\begin{equation}\label{eq:not-point-mass-inCLT} \sum_{j\geq 2} |a_j|(\eps/32)^j > b, \end{equation}
then\footnote{Here, the implicit constant Corollary~\ref{cor:CLT} can be taken to be $2^{3255}$.} 
$$\sup_{t \in \R} |\PP(X \leq t) - \PP(Z \leq t)| = O\left(\frac{1}{\eps \sigma}\right) \,,$$ where $Z \sim N(0,1)$ is a standard normal.
\end{corollary}

\subsection{Proof of Theorem~\ref{thm:Polysector} } 
We are now ready to prove our main theorem on random variables with roots avoiding a sector.

\vspace{4mm}

\noindent \emph{Proof of Theorem~\ref{thm:Polysector}.} Let $X \in \{0,\ldots,n\}$ be a random variable for which its probability generating function $f_X$ has no roots
in the sector $S(\delta)$. This means that its logarithmic potential $u(z) = u_X(z)$ is a weakly-positive, symmetric and harmonic function on $S(\delta)$.
Also, since $f_X$ is a polynomial, we have that $u(z) = O(\log|z|)$ and $u(1/z) = O(\log|z|)$ as $z \rightarrow \infty$. Finally note that we may assume $\s >0$,
otherwise the statement of Theorem~\ref{thm:Polysector} is meaningless.

We first look to apply Lemma~\ref{lem:b-decreasing} to show that $u$ is decreasing in a neighborhood of $1 \in \C$; that is, $b$-decreasing for $b=0$.
For this, note that for all $R> 1$ we have that $S_R(\delta/2) \subseteq S(\delta/2)$ and so $u$ is harmonic \emph{in a neighborhood} of $S_R(\delta/2)$.
Set $r := 2$ and we check, in accordance with \eqref{equ:bcondition}, that
\[ \left( \frac{2}{R}\right)^{-1/\delta}\max_{z \in S_R^{\ast}(0,\delta) } |u(z)| = O\left( R^{-1/\delta} \log R\right) \rightarrow 0,
 \] 
 as $R \rightarrow \infty$, due to the growth condition on $u$. Thus we may apply Lemma~\ref{lem:b-decreasing} to learn that $u$ is $b$-decreasing in $S_{2}(0,\delta/2)$ for every $b \geq 0$, and therefore is \emph{decreasing}. 

We now look to apply Corollary~\ref{cor:CLT} to finish the proof of the theorem. For this
we only have to check the condition at \eqref{eq:not-point-mass-inCLT}, which easily follows from the fact that $\s >0$.
Since $u$ is $0$-decreasing and harmonic in $B(1,\delta/4)$, we may apply Corollary~\ref{cor:CLT} with $\eps = \delta/4$ to finish the proof. \qed
\vspace{4mm}

\subsection{Proof of Theorem~\ref{thm:logn}}

The proof of Theorem~\ref{thm:logn} is similar to Theorem~\ref{thm:Polysector}. In the proof of Theorem~\ref{thm:logn} we work in a tiny truncated sector which we place in the zero-free ball $B(1,\eps)$ and then estimate 
$u_X(z)$ on the ends $S^{\ast}_R(\delta)$ of the boundary $\p S_R(\delta)$. This estimate is the content of the following lemma.

\begin{lemma} \label{lem:size-u-in-ball} For $\delta \in (0,1/2)$, let $X \in \{0,\ldots,n\}$ be a random variable
for which $f_X$ has no roots in the ball $B(1,\delta)$.  For $\eps\in (0,\delta/4)$, we have 
\[ \max_{z \in S^{\ast}_R(\eps)} |u_X(z)| \leq 7n \log(4/\delta). \]\end{lemma}
\begin{proof} First, put $S = S_R(\eps)$ and $S^{\ast} = S_R^{\ast}(\eps)$. We use the expansion of $u = u_X$ in terms of the roots of $f_X$ (see \eqref{equ:rootFormOfu}, Section~\ref{sec:defs})
	\begin{equation} \label{eq:root-exp-of-u-inproof}
 u_X(z) = \sum_{|\zeta| < 1} \log\left|1 - z^{-1}\zeta \right| + \sum_{|\z| \geq 1 }\log\left|1 -z \zeta^{-1} \right| + c_X + N_X\log|z|,
	\end{equation}
	where the sums are over the roots of $f_X$ and $N_X = |\{ \z : |\z| < 1 \}|$.

 For $z\in S$, let $\l \in \{ \z^{-1} \}_{|\z|\geq 1} \cup \{ \z \}_{|\z| \leq 1}$;
 we note the inequality
\begin{equation} \label{eq:inequal} \left|\log\left|1-\l z\right|\right| \leq \max\{ \log 4/\delta , |\l z| \}  ,
\end{equation} which holds because $\log\left|1-\l z\right| \leq |\l z|$ and 
\[ -\log\left|1-\l z\right| \leq -\log(|z|\cdot|1 -\l| - |1 - z|)  \leq \log 4/\delta.\]
So, using that $u(1) = 0$, we have 
\begin{align*}
 |c_X| &\leq \sum_{|\zeta| < 1} \big|\log\left|1 - \zeta \right| \big| + \sum_{|\z| \geq 1 }\big|\log\left|1 - \zeta^{-1}\right|\big|  \\
 &\leq n (\log 4/\delta),
 \end{align*} since $\delta < 1$.  Since $|N_X| \leq n$, we have $|N_X \log|z| | \leq n \log|z|.$ 
 Now applying the triangle inequality and \eqref{eq:inequal} to \eqref{eq:root-exp-of-u-inproof} yields 
\[ | u(z)| \leq \sum_{|\zeta| < 1} \frac{|\zeta|}{|z|} + \sum_{|\z| \geq 1 } \frac{|z|}{|\zeta|} + n\log 4/\delta + |c_X| + |N_X \log|z| |\leq n\left(|z| + |z|^{-1}+3\log(4/\delta)\right).
\]
Putting these together with the fact that $|z|^{-1},|z| \leq 2$ (since $z\in S$) implies \[ |u(z)| \leq 7n \log(4/\delta), \] for $z \in S$.
\end{proof}

\vspace{5mm}

\noindent\emph{Proof of Theorem~\ref{thm:logn} : } Let $X \in \{0,\ldots,n\}$ be a random variable for which its probability generating function $f_X$ has no zeros in the ball $B = B(1,\delta)$, for $\delta \in (0,1/2)$.
Note that we may assume that $\s > 2^{11}\delta^{-1}\log n$ otherwise (\ref{equ:lognCloseToNorm}) is trivial, with implicit constant $2^{11}$. Since
$X \in \{0,\ldots,n\}$ we have $\s \leq n$, and therefore we may assume that $\delta > 1/n$, otherwise the theorem is trivial. 

Using the fact that there are no roots in $B$ implies that $u(z) = u_X(z)$ is harmonic in the ball $B$. We now work in a thin, truncated sector inside of $B$. In particular, we set $\eps = \delta/(64\log n)$, $R = 1+\delta/4$, $r = 1+\eps$ and work in the sector $S := S_R(\eps)$. Since $S_R(\eps) \subseteq B$, we have that $u$ is harmonic in a neighborhood of $S_R(\eps)$.

Now choose $b=1$ and, looking to apply Lemma~\ref{lem:b-decreasing}, we verify \eqref{equ:bcondition}. First note that for $n \geq 3$,
\[  \frac{1}{\eps}\log R/r  \geq (64\log n) \delta^{-1}\log(1+\delta/4)  - 1 \geq 7\log n. 
\] and therefore
\begin{equation} \label{eq:b-dec-check} \exp\left( - \frac{1}{\eps}\log R/r \right) \max_{z \in S_R^{\ast}(0,\eps) } |u(z)| \leq  \frac{1}{n^7}\left(7n\log 4/\delta\right) \leq \frac{7\log 4n }{n^6}, \end{equation} using Lemma~\ref{lem:size-u-in-ball} and $\delta > 1/n$. Thus \eqref{eq:b-dec-check} is at most $1\cdot (3/8)$, for $n \geq 3$, and thus we may apply Lemma~\ref{lem:b-decreasing} to conclude that $u$ is $1$-decreasing in $B(1,\eps)$.

We now look to apply Corollary~\ref{cor:CLT}. Since $u$ is $1$-decreasing, weakly-positive and harmonic in $B(1,\eps)$, we only need to check condition \eqref{eq:not-point-mass-inCLT} for $b=1$. This is easily done as 
\[ \sum_{j\geq 2} |a_j|(\eps/32)^j \geq \left(\frac{\eps \s}{32}\right)^2 \geq \left( \frac{\s \delta }{2^{11} \log n} \right)^2 > 1 = b, 
\] where we have used the assumption that $\s > 2^{11}\delta^{-1}\log n$, since the statement is trivial otherwise.  We now apply Corollary~\ref{cor:CLT} with 
$\eps = \delta/(64 \log n)$ to complete the proof.  \qed

\section{Multivariate central limit theorem for strong Rayleigh measures} \label{sec:CLT-for-stable}
In this section we prove that strong Rayleigh distributions satisfy central limit theorems. 
For a $d \times d$ positive semi-definite matrix $A$ and a vector $\mu \in \R^{d}$ we define $N(\mu,A)$ to be the multivariate Gaussian random variable with mean $\mu$ and covariance matrix $A$.  To prove that a random variable $X \in \R^d$ is a multivariate normal distribution, we show that ``many'' of its one dimensional projections 
$\langle X, v\rangle$ are Gaussian. We will then apply a variant of the famous Cram\'{e}r-Wold theorem which will allow us to conclude that 
$X$ itself must be a \emph{multivariate} Gaussian.

The key connection between stable polynomials and univariate polynomials that have no roots in a sector comes from the following fundamental observation, first made by Ghosh, Liggett and Pemantle \cite{GLP}.

\begin{lemma} \label{lem:roots-twisted-gen-function} Let $Y \in \Z^d$ be a finitely supported random variable with real-stable probability generating function $f_Y$. If $v = (v_1,\ldots,v_d) \in \Z_{\geq 0}^d$ then the probability generating function of $\langle v ,Y \rangle $ has no zeros in the sector $S(\pi/\|v\|_{\infty})$.
\end{lemma}
\begin{proof}
Let $f_Y$ be the probability generating function of $Y$, let $f_v$ be the probability generating function of $Y(v) := \langle v ,Y \rangle $ and let $\z$ be a root of $f_v$. First note that 
\begin{equation} \label{eq:f_v-form} f_v(z) = f_Y(z^{v_1},\ldots,z^{v_k}).
\end{equation} Now write $\z = re^{i\t}$, for $r >0$, and $\t \in [-\pi,\pi]$. Since $f_v \in \R[z]$, we may additionally assume that $\t \in [0,\pi]$, by possibly replacing $re^{i\t}$ with its conjugate. From \eqref{eq:f_v-form}, we see that $(r^{v_1}e^{iv_1\t},\ldots, r^{v_d}e^{iv_d\t})$ is a root of $f_Y(z_1,\ldots,z_d)$ and therefore, by the real-stability of $f$, there is some $i \in [d]$ with $\sin(v_i\t) < 0$ and therefore $\t \in (\frac{\pi}{v_i},\pi] \subseteq (\frac{\pi}{\|v\|_{\infty}}, \pi]$.\end{proof}

\vspace{4mm}

Lemma~\ref{lem:roots-twisted-gen-function} allows us to use Theorem~\ref{thm:Polysector} to show that all the projections of a strong Rayleigh distribution are 
approximately normal, in non-negative integer directions. Note that we will take the degenerate normal $N(0,0)$ to also be a normal random variable: indeed, this measure is simply the point mass at $0$.

\begin{lemma} \label{lem:Proj-agree} For each $n \geq 1$, let $X_n \in \{0,\ldots,n\}^d$ be a strong Rayleigh distribution with mean $\mu_n$, covariance matrix $A_n$ and maximum variance $\s_n^2$.
Put $X^{\ast}_n = (X_n - \mu_n)\s^{-1}_n$. If $\s_n \rightarrow \infty$ and $\s_n^{-2}A_n \rightarrow A$, then for all $v \in \Z_{\geq 0}^d$, we have that 
$$ \langle X^{\ast}_n, v \rangle \to N(0, v^T A v )\,, $$
in distribution.
\end{lemma}
\begin{proof}

Let us put $Y_n(v) := \langle X_n, v \rangle = v_1X_1 + \cdots + v_dX_d $. Note that $\EE\, Y_n(v) = \langle v,\mu_n \rangle$ and that  
\begin{equation} \label{eq:var_form}
 \Var(Y_n(v)) = v^TA_n v. 
\end{equation} From Lemma~\ref{lem:roots-twisted-gen-function}, we see that the probability generating function $f_{Y_n(v)}$ of $Y_n(v)$ has no roots 
in the sector $S(\delta)$ where $\delta = \pi/\|v\|_{\infty}$. 

There are two cases: when $v$ is in the null-space of $A$ and when $v$ is not in the null space of $A$. Let us first assume that $Av \not= 0$.
In this case we have 
\[ \lim_n \Var(Y_n(v))\s_n^{-2} = \lim_n v^T\left(\s_n^{-2}A_n\right)v = v^T A v \neq 0,
\] and in particular $\Var(Y_n(v)) \rightarrow \infty$. Thus we may apply our central limit theorem for random variables avoiding a sector, Corollary~\ref{cor:sector-CLT},
to see that 
\[ \frac{ Y_n(v) - \EE Y_n(v) }{(v^TA_nv)^{1/2}} \rightarrow N(0,1) 
\] and therefore
\[ \frac{\s_n}{(v^TA_nv)^{1/2}} \langle X^{\ast}_n, v \rangle  =  \frac{ Y_n(v) - \EE Y_n(v) }{(v^TA_nv)^{1/2}} \rightarrow N(0,1).
\] Since $(v^TA_nv)^{1/2}\s_n^{-1}$ tends to a constant as $n \rightarrow \infty$, it follows that 
\[ \langle X^{\ast}_n, v \rangle  \rightarrow N(0,  v^T A v), 
\] as desired.

In the other case, we have that  $Av = 0 $ and therefore $\Var(Y_n(v))\s_n^{-2} \rightarrow 0$. So, for all $x >0$, we may apply Chebyshev's inequality to see that
\[ \PP\left( |Y_n(v) - \EE Y_n(v)|\geq x \s_n \right) \leq \Var(Y_n(v))(x\s_n)^{-2} = o(1).  
\] This simply means that $(Y_n(v) - \EE Y_n(v))\s_n^{-1}$ tends to a point mass at zero, in distribution. 
So trivially,
\[  \langle X^{\ast}_n, v \rangle \rightarrow N(0,0) = N(0,v^TAv),
\] as desired.

\end{proof}

To ``lift'' this information about the projected random variables, we appeal to a theorem of Cuesta-Albertos, Fraiman and Ransford \cite{Cuesta-Albertos}, which will allow us to conclude that the distribution of our multivariate random variable is approximately normal from the fact that ``many'' of its projections are normal.
To properly state this result, we let $\nu$ be a Borel probability measure on $\R^d$ and, for $v \in \R^d$, we define the measure $\nu_v$ to be the ``projected'' measure on $\R$ defined by
\[ \nu_v(B) = \nu\left(\left\lbrace x \in \R^d : \langle v , x\rangle \in B \right\rbrace \right),
\] for every Borel set $B \subseteq \R$. If $\tilde{\nu}$ is another Borel probability measure on $\R^d$, we define $\Pi(\nu,\tilde{\nu}) \subseteq \R^{d}$ to be the set of $v \in \R^{d}$ for which $\tilde{\nu}_v = \nu_v$. In this notation, the classical Cram\'{e}r-Wold Theorem \cite{Cramer-Wold} says that if $\nu,\tilde{\nu}$ are Borel probability
measures such that $\Pi(\nu,\tilde{\nu}) = \R^d$ then $\nu = \tilde{\nu}$. Cuesta-Albertos, Fraiman and Ransford have sharpened this result by showing
that it is enough for $\Pi(\nu,\tilde{\nu})$ to not be contained in the zero-set of a polynomial. We shall only make use of the following corollary of this theorem (stated here as Theorem~\ref{thm:CAFR}), which we derive in Appendix \ref{app:proofs}.

\begin{corollary} \label{cor:CAFR}
For $d \geq 1$, let $A$ be $d \times d$ positive semi-definite matrix and let $\nu_A$ be the Gaussian distribution on $\R^d$ with covariance matrix $A$ and mean zero. If $\nu$ is a measure
for which $\Pi(\nu_A,\nu) \supseteq \Z_{\geq 0}^d$ then $\nu_A = \nu$.
\end{corollary}

Recall that a sequence of Borel probability measures $\nu_n$ on $\R^d$ is said to be \emph{tight}, if for every $\eps >0$ there exists $R = R(\eps) >0$, so that the ball $B(0,R)$ satisfies $\nu_n(B(0,R)) > 1-\eps$, for all sufficiently large $n$. 
We need the following basic fact, proved in Appendix \ref{app:proofs}.

\begin{fact} \label{fact:tight} For each $n \geq 1$, let $X_n \in \R^d$ be a random variable, with finite mean $\mu_n$, covariance matrix $A_n$ and maximum variance $\s_n^2 < \infty$.
If $\nu_n$ is the law of $X_n^{\ast} := (X_n - \mu_n)\s_n^{-1}$ then $\nu_n$ is a tight sequence of measures.
\end{fact} 

The key fact we use about tight sequences of measures is that they have subsequential limits; the following easy fact can be proved by a standard 
diagonalization argument.

\begin{fact} \label{fact:compactness}
For $d \in \N$ and each $n\geq 1$, let $\nu_n$ be a Borel probability measure on $\R^d$. If $\{\nu_n\}_n$ is a tight sequence then there exists a
Borel probability measure $\nu$ and a subset $S \subseteq \N$ so that $\nu_{n} \rightarrow \nu$, in distribution, as $n \in S$ and $n\rightarrow \infty$.
\end{fact}

From here we can easily put the pieces together to finish the proof of Theorem~\ref{thm:stable_clt}.

\vspace{4mm}

\noindent \emph{Proof of Theorem~\ref{thm:stable_clt}. } For each $n \geq 1$, let $X_n \in \{0,\ldots,n\}^d$ be a random variable with mean $\mu_n$, covariance matrix $A_n$,
maximum variance $\s_n^{2}$ and let $\nu_n$ be the law of $X_n^{\ast}:= (X_n-\mu_n)\s_n^{-1}$. We have that $\s_n^{-2}A_n\rightarrow A$, for some (non-zero) matrix $A$. 

Let $\nu_A$ denote the law of $N(0,A)$; we show that every subsequence has a further subsequence that converges to $\nu_A$, which is enough to conclude that $\nu_n \rightarrow \nu_A$. For this, let $S \subset \N$; by Facts \ref{fact:tight} and \ref{fact:compactness}, we may find $S' \subset S$ so that along $S'$ we have $\nu_n \to \nu'$ for some measure $\nu'$.  Convergence in distribution together with Lemma \ref{lem:Proj-agree} imply that $\Pi(\nu_A,\nu') \supseteq \Z_{\geq 0}^d$.  Corollary \ref{cor:CAFR} then implies $\nu' = \nu_A$. This completes the proof. \qed

\begin{remark}
We point out that our results here easily generalize beyond real stable polynomials to other situations 
where $f_X$ satisfies a certain ``half-plane property''. To take a well-known example, we say that a polynomial is \emph{Hurwitz stable} if it has no roots 
in
\[ \{ (z_1,\ldots,z_d) \in \C^d : \Re(z_i) > 0, \textit{ for all } i \}. \] 
Our work here implies a version of Theorem~\ref{thm:stable_clt} in the case that $f_X$ is Hurwitz stable. In fact, the only point to check is Lemma~\ref{lem:roots-twisted-gen-function} and the rest of the proof proceeds in the same way. 

More generally, let $\phi \in (0,2\pi)$ and define 
\[ H_{\phi} := \{ (z_1,\ldots,z_d) \in \C^d : \arg(z_i) \in  [0,\phi], \textit{ for all } i \}.
\] We say that a polynomial $f$ is 
$H_{\phi}$-\emph{stable} if it has no roots in $H_{\phi}$. It is not hard to see that our results imply a central limit theorem for a sequence of 
random variables $X_n$ when the $f_{X_n}$ are $H_{\phi}$-stable polynomials and $\s_n \rightarrow \infty$.
\end{remark}

\section{Sharpness of results} \label{sec:sharpness}

In this section we show that our quantitative results, Theorems~\ref{thm:logn} and \ref{thm:Polysector}, are sharp up to the implied constants. From this it will also follow that the conditions in the limit theorems, Corollaries~\ref{cor:ball-CLT} and \ref{cor:sector-CLT}, are best-possible.

Our constructions follow from a few simple observations. This first observation gives a cheap bound on the discrepancy between a discrete random variable and the standard normal distribution. Recall that we use the notation $X^{\ast}$ to denote $(X -\mu)\s^{-1}$, for a random variable $X$.

\begin{observation}\label{obs:discrete-discrep} Let $X \in \Z$ be a random variable with mean $\mu <\infty $ and standard deviation\footnote{The $2^{-3}$ is an arbitrary choice, we just needed a sufficiently small number for the application of Theorem~\ref{thm:Sector-is-sharp}.}
 $ \s \in [2^{-3}, \infty)$. Then 
\[ \sup_{t \in \R} |\PP(X^{\ast} \leq t) - \PP(Z \leq t)| \geq e^{-16}/\sigma, \]
where $Z \sim N(0,1)$.
\end{observation}
\begin{proof}
Note that since $X \in \Z$, we have $X^{\ast} := (X - \mu) \s^{-1} \in \frac{1}{\s}(\Z - \mu)$.  Find values $a, b \in \Z - \mu$ so that $b - a = 1$, $a \leq 0$ and $b \geq 0$.  Then $\PP(X^\ast \in \frac{1}{\s}(a,b)) = 0$ while 
\[ \PP\left(Z \in \frac{1}{\s}(a,b) \right) = \frac{1}{(2\pi)^{1/2}}\int_{a/\s}^{b/\s} e^{-s^2/2} ds \geq  \frac{1}{(2\pi)^{1/2}}\int_{0}^{1/2\s} e^{-s^2/2} ds  \geq \frac{e^{-16}}{\s}, 
\] where $Z \sim N(0,1)$ and we have used that one of $|a|$ or $|b|$ must be at least $1/2$. This allows us to obtain a lower bound on the maximum discrepancy between the two cumulative distribution functions. We have
\[ 2 \cdot \sup_{t \in \R}|\PP(X^{\ast} \leq t ) - \PP(Z \leq t)|\geq |\PP(X^{\ast} \leq b/\s) - \PP(Z\leq b/\s)| + | \PP( X^{\ast} \leq a/\s ) -  \PP( Z \leq a/\s )|,
\] which is at least 
\[   \PP\left( Z \in \frac{1}{\s}(a,b)\right) - \PP\left( X^{\ast} \in \frac{1}{\s}(a,b)\right) \geq e^{-16}/\s,
\] by the triangle inequality.
\end{proof}

\vspace{4mm}

The next basic observation records a key ``trick'' in our constructions. It says that the transformation $X \mapsto k \cdot X$ does not change the maximum discrepancy with a normal. However, the standard deviation \emph{increases} as $\s(k \cdot X) = k\s(X)$.

\begin{observation} \label{obs:mult-by-k}
Let $Y \in \Z$ be a random variable with finite mean $\mu $ and standard deviation $\s \in [2^{-3}, \infty)$. For $k > 0$, let $X = k \cdot Y$.
Then
\[ \sup_{t \in R} |\PP(X^{\ast} \leq t) - \PP(Z \leq t)| \geq \frac{ck}{\sigma(X)}, 
\] where we can take $c = e^{-16}$.\end{observation}
\begin{proof}
Note that $\s(X) = k\s(Y)$ and $\EE X = k \EE Y$ thus  
\[ X^{\ast} = (X - \EE\, X)\s(X)^{-1} = (Y - \EE Y)\s(Y)^{-1} = Y^{\ast}, \] and so 
\[ \sup_{t \in \R}|\PP(X^{\ast} \leq t ) - \PP(Z \leq t)| = \sup_{t \in \R}|\PP(Y^{\ast} \leq t ) - \PP(Z \leq t)|. 
\] Thus, applying Observation~\ref{obs:discrete-discrep} to $Y$ yields
\[\sup_{t \in \R}|\PP(X^{\ast} \leq t ) - \PP(Z \leq t)|  \geq \frac{e^{-16}}{\s(Y)} = \frac{ck}{\s(X)}, \] as desired.
\end{proof}

\vspace{4mm}

Our constructions for both theorems in this section are achieved simply by applying Observation~\ref{obs:mult-by-k} to an appropriate ``seed'' random variable. For Theorem~\ref{thm:Sector-is-sharp}, we make use of a simple class of random variables. If $\t \in [\pi/2,\pi]$, the polynomial 
\[P_{\rho,\t}(z) = (z-\rho e^{i\t})(z-\rho e^{-i\t}) = x^2 -2\rho(\cos \t) x + \rho^2
\] has non-negative coefficients and therefore $P_{\rho,\t}(z)(P_{\rho,\t}(1))^{-1}$ is the probability generating function of a random variable, which we shall denote by $Y_{\rho,\t}$.

Now note that for each fixed $\t$, as $\rho \geq 1$ increases, $\Var(Y_{\rho,\t})$ decreases as a continuous function of $\rho$.  Further, each random variable is non-degenerate for $\rho \in[1,\infty)$, implying $\Var(Y_{\rho,\t}) > 0$.  Since we also have $\lim_{\rho \to \infty} \Var(Y_{\rho,\t}) = 0$, there exists some $a(\t)>0$ so that $\{ \Var(Y_{\rho,\t}) \}_{\rho \geq 1} \supseteq [0,a(\t)]$.

\begin{theorem} \label{thm:Sector-is-sharp}
For every $\delta \in (0,\pi]$ and $\s > 0$ with $\delta \sigma \geq 1$, there exists a random variable $X \in \Z_{\geq 0}$, which is supported on finitely many integers, with standard deviation $\s$ and probability generating function $f_X$ for which $\delta = \min_{\z: f(\z) = 0}|\arg(\z)|$ and
\[\sup_{t \in \R} | \PP( X^{\ast} \leq t ) - \PP(Z \leq t)| \geq \frac{c}{\delta \s},\]
where we can take $c = e^{-16}$. 
\end{theorem}
\begin{proof}
Let $(\s,\delta)$ be given. As $\bigcup_{j\geq 0} [\pi/2^{j+1},\pi/2^j] = (0,\pi]$, we may write $\delta =  \t/k$, for some $k \in \N$ and $\t \in [\pi/2,\pi]$
and note that $1 \leq \delta \s = (\t \s)/k$.
We start by constructing a random variable $Y$ with standard deviation $\s/k \geq 2^{-3}$ and $\min_{\z}|\arg(\z)| = \theta$. We then finish by applying Observation~\ref{obs:mult-by-k}.

For $\rho,m$ to be chosen later, let $Y_i$ be independent copies of $Y_{\rho,\t}$ and let 
\[ Y = \sum_{i=1}^m Y_i. 
\] Of course, $\s(Y) = m^{1/2}\s(Y_{\rho,\t})$ and thus, from the discussion that precedes Theorem~\ref{thm:Sector-is-sharp}, we may choose $m,\rho$ so that $m^{1/2}\s(Y_{\rho,\t}) = \s/k$.
Moreover, every root $\z$ of the probability generating function $f_Y = \left(f_{Y_{\rho,\t}}\right)^m$ of $Y$ has $\arg(\z) \in \{-\t,\t\}$.

Finally, set $X = k\cdot Y$. The probability generating function of $X$ is $f_X(z) = f_Y(z^k)$ and thus the roots $\z$ of $f_X$ satisfy 
\[ \arg(\z) \in \left\lbrace \pm \t/k + \frac{2\pi \ell}{k} \mod 2\pi : \ell \in \{0,\ldots,k-1\}\right\rbrace
\] and therefore $\min_{\z : f_X(\z) = 0 } | \arg(\z)| = \t/k = \delta$. 
 From Observation~\ref{obs:mult-by-k}, we have that 
\[  \sup_{t \in \R} |\PP(X^{\ast} \leq t) - \PP(Z \leq t)| \geq \frac{e^{-16} k}{\s(X)} \geq \frac{e^{-16}}{\delta \s(X)},
 \] where the last inequality follows from the fact that $k \delta = \t \in [\pi/2,\pi]$. This completes the proof. \end{proof}

The following shows that Theorem~\ref{thm:logn} is sharp. Here, we apply Observation~\ref{obs:mult-by-k} to a sum of Bernoulli random variables.
 
\begin{theorem}
For $n\geq 1$, $\delta > 0$ and $\s^2 \in [1,n^{0.9})$ satisfying $\frac{\log n}{\delta \s} \leq 1$ there exists a random variable $X \in \{0,\ldots, n\}$
with standard deviation $\s$ so that $\min_{\z : f_X(\z) = 0}|1-\z| \geq \delta $ and 
\[ \sup_{t \in \R} |\PP(X^{\ast} \leq t) - \PP(Z\leq t) | \geq \frac{c\log n}{\delta \s}. \]
\end{theorem}

\begin{proof}
	Let $\{ Y_i\}_{i\geq 1}$ be independent and identically distributed Bernoulli random variables where $p := \PP(Y_i = 1)$ will be chosen later. 
	Of course, we have that $\Var(Y_i) = p(1-p)$ and the probability generating function of $Y_i$ is $pz + (1-p)$. We set  
	\[ Y = \sum_{i=1}^{ \lfloor n/k \rfloor} Y_i, 
	\]
	with $k := \lfloor \log n / (100 \delta) \rfloor$ and note that $\Var(Y) =  \left\lfloor\frac{n}{k}\right\rfloor p(1-p)$. We define $X := kY$ and set 
	$p = n^{-\alpha}$ with $\alpha \in [0.01,1)$ to be chosen later.
	To apply Observation~\ref{obs:mult-by-k} to $X$, we require that $\Var(Y) \geq 1/8$; and so we impose the condition 
	\[ n^{1-\alpha} \geq \log n/(100 \delta)\] to guarantee this.
	Now,
	\[ \Var(X) = k^2 \left\lfloor\frac{n}{k}\right\rfloor  p(1-p),
	\] is a continuous function of $\alpha$ as $p = n^{-\alpha}$. Also
	 \[ \Var(X) \leq knp(1-p) \leq \frac{\log n}{100\delta}n^{1-\alpha}(1-n^{-\alpha}). \] and
	 \[ \Var(X) \geq knp(1-p) - k^2p(1-p) \geq \frac{\log n}{200\delta}n^{1-\alpha}(1-n^{-\alpha}),
	 \] where the last line holds when $n \geq \log n/(200\delta) \geq k/2$, which always holds for us as $n \geq \s \geq \log n/\delta$, by hypothesis.
	 So as $\alpha \in [0.01,1)$ varies subject to $n^{1-\alpha} \geq \log n/(100 \delta)$, $\Var(X)$
	ranges over a set containing the interval $[ \delta^{-2}(\log n)^2,n^{0.9}]$ and since $\frac{\log n}{\delta \s} \leq 1$ and $\s^2 < n^{0.9}$ we may select $\alpha \in [0.01,1)$ so that $\Var(X) = \s^2$.
	
	Now note that $\deg(f_X) = k\lfloor n/k \rfloor \leq n$ and thus $X \in \{0,\ldots,n\}$. Since $f_X(z) = f_Y(z^k)$, the roots $\z$ of $f_X$ are of the form
	$\z =  \beta \left(\frac{1-p}{p}\right)^{1/k}$, where $|\beta| =1$, which allows us to bound
	\[ \min_{\z}|1-\z|\geq |1-e^{\frac{\alpha\log n}{k}}| = |1-e^{\alpha \log n /\lfloor \log n / (100 \delta) \rfloor}| \geq \delta.
	\]
	Applying Observation~\ref{obs:mult-by-k}, we see that 
	\[ \sup_{t \in \R} |\PP(X^{\ast} \leq t) - \PP(Z\leq t) | \geq \frac{ck}{ \s(X) } \geq \frac{c\log n}{100 \delta \s} \geq \frac{C \log n}{\delta \s}\,,
	\] as desired.\end{proof}

\section{General distributions} \label{sec:general-distributions}

In this brief section we discuss how to apply our results to random variables that take values in $\R$, rather than just in $\{0,\ldots,n\}$. 
In short, everything for Theorem \ref{thm:Polysector} extends rather naturally, but a few extra complications arise. 

The first task is to fix an appropriate notion of the probability generating function of $X$. Luckily, there is already a standard definition 
in this situation. First set $z^r := \exp(r\log z)$, for all $r \in \R $, where ``$\log$'' denotes the standard branch of the logarithm; then define
\[ f_X(z) := \EE_X\, z^X, 
\] for all $z \in \C \setminus \R_{\leq 0}$, to be the probability generating function of $X$.  

We now happen upon a feature of the more general set-up: $f_X$ does not necessarily \emph{exist} for all $z$ and therefore it may not make any sense to discuss the zeros of $f_X(z)$ at all. To ensure the existence of $f_X$, for all $z \not\in \R_{\leq 0}$,
it is enough to impose the condition $f_Z(\rho) < \infty$ for all $\rho >0$.  With this assumption in hand, Morera's theorem shows that $f$ is analytic as well (see Appendix \ref{app:proofs} for a proof):

\begin{lemma} \label{lem:analytic}
	Let $X \in \R$ be a random variable and let $f_X$ be its probability generating function. If $f(\rho) < \infty$ for all $\rho >0$ then
$f_X(z)$ is analytic in $\C \setminus \R_{\leq 0}$.
\end{lemma}

A second subtlety concerns the asymptotic growth of the logarithmic potential $u_X(z) := \log |f(z)|$, for $|z|$ very large and very small. 
For $\kappa, \delta >0$ we say that $u$ satisfies the $(\kappa,\delta)$-\emph{growth condition} if we have
\begin{equation} \label{equ:GrowthCondition} \lim_{|z| \rightarrow \infty} \frac{|u_X(z)|}{|z|^{\kappa}} = 0  \textit{ and }  
\lim_{|z| \rightarrow \infty} \frac{|u_X(1/z)|}{|z|^{\kappa}} = 0, \end{equation} where the limits are taken with $z \in S(\delta)$.

In previous sections, we could ignore \eqref{equ:GrowthCondition}, as $u_X$ trivially satisfies the $(\kappa,\delta)$-growth condition for all $\kappa >0$
when $f_X$ is a polynomial. Here, however, we are forced to take the rate of growth into account, as it directly affects the convergence to a normal distribution.

We now state our main general theorem for zero-free sectors of probability generating functions.

\begin{theorem}\label{thm:General} For $\delta  >0$, and $\kappa >0$ let $X \in \R$ be a random variable with probability generating function $f_X$
for which $f(\rho)$ is defined for all $\rho \in \R_{\geq 0}$. If $u_X$ satisfies the $(\kappa,\delta)$-growth condition and $f_X$ has no zeros in $S(\delta)$
then \footnote{The implicit constant may be taken to be $2^{3258}$.}
\begin{equation} \label{equ:sectorCloseToNorm-general} \sup_{t \in \R}| \PP(X^{\ast} \leq t) - \PP(Z \leq t) | = O\left( \frac{\max\{\delta^{-1},\kappa\}}{\sigma}\right), \end{equation}
where $Z \sim N(0,1)$. \end{theorem}

Again, this theorem is sharp with respect to the dependence on $\kappa,\delta$ and $\s$, as we shall see in Subsection~\ref{subsec:gen-sharp}.

\subsection{Proof of Theorem~\ref{thm:General}}

We first notice that many of the properties of the logarithmic potential of $X$ easily carry over to this more general setting. Indeed,
if $f_X$ is zero-free in the sector $S(\delta)$ then $u(z)$ is harmonic in this sector. Also $u(z)$ is symmetric, and weakly positive. 
We also have that $f_X(1) = 1$ and therefore $u_X(1) = 0$. With these observations at hand, we may prove Theorem~\ref{thm:General}
as we proved Theorem~\ref{thm:Polysector}

\vspace{4mm}

\noindent \emph{Proof of Theorem~\ref{thm:General}.} Let $X \in \R$ be a random variable with probability generating function $f_X$,
satisfying $f_X(\rho) < \infty$, for all $\rho >0$; that is zero-free in the sector $S(\delta)$; and so that $u_X$ satisfies the $(\kappa,\delta)$-growth condition. 
By the discussion above, we know that the logarithmic potential $u = u_X$ is harmonic, symmetric and weakly-positive in $S(\delta)$.
Also note that we may assume that $\s >0$, otherwise the statement of the Theorem~\ref{thm:General} is meaningless.

We now choose $\eps = \min\{ \delta/2,\frac{1}{2\kappa} \}$ and note that $u$ is a weakly-positive, symmetric and harmonic function on the smaller region $S(\eps)$.
Now, looking to apply Lemma~\ref{lem:b-decreasing}, we set $r := 2$ and note that 
\begin{align}\label{eq:b-decreasing-check-gen}
\left(\frac{2}{R} \right)^{1/\eps}\max_{z \in S_R^{\ast}(\eps) } |u(z)| = O\left(R^{-1/\eps + \kappa} \right) = O(R^{-\kappa}) \to 0 \end{align} as $R \rightarrow \infty$. 
Thus, we may apply Lemma~\ref{lem:b-decreasing} to learn that $u$ is decreasing in $S(\eps/2)$.  Lemma \ref{lem:planar-geometry3} then implies that $u$ is decreasing in $B(1,\eps/4)$. Since $\s>0$, $u$ satisfies the conditions of Corollary~\ref{cor:CLT}, which finishes the proof. \qed

\subsection{Proof of the sharpness of Theorem~\ref{thm:General}} \label{subsec:gen-sharp}

\begin{theorem} \label{thm:growth-is-sharp}
	Let $\kappa$, $\delta \in (0,\pi)$ and $\sigma$ be so that $\sigma \cdot \min\{\delta,1/\kappa\} \geq 1$.  Then there exists a random variable $X \in \Z$ with standard deviation $\s$ so that $u_X$ is harmonic in $S(\delta)$, $u_X$ satisfies the $(\kappa,\delta)$-growth condition, and 
\begin{equation}\label{eq:growth-sharp-disc} \sup_{t \in \R} |\PP(X^\ast \leq t) - \PP(Z \leq t) | \geq \frac{c}{\s}\cdot \max\{\kappa,\delta^{-1}\}\,.\end{equation}
\end{theorem}
\begin{proof}
	If $\delta \leq 1/\k$, we apply Theorem \ref{thm:Sector-is-sharp} to obtain a random variable $X \in \Z_{\geq 0}$ which has finite support
	and satisfies \eqref{eq:growth-sharp-disc}. Here, $f_X$ is a polynomial so $\log|f_X| = O(\log |z|) = O(|z|^{\kappa})$.  
	
	In the case of $\delta > 1/\k$, let $Y$ be the Poisson random variable with mean $4\s^2 / \kappa^2$.  Then $Y \in \Z$ with $\s(Y) = 2\s/\kappa$.  Set $X = (\kappa/2) \cdot Y$  and note $\sigma(X) = \s$ and $$u_X(z) = \frac{4\s^2}{\kappa^2}\left( z^{\kappa/2} - 1\right)\,, $$ which is harmonic in $S(\delta)$ and satisfies the specified growth conditions.  Applying Observation \ref{obs:mult-by-k} completes the proof. 
\end{proof}

\section{Acknowledgments}

The authors would like to thank Robin Pemantle for introducing us to these problems and Robert Morris for comments on a draft. The second-named author is especially grateful to Robin Pemantle and the mathematics department at the University of Pennsylvania for hosting him on numerous occasions while this research was being carried out.

\bibliographystyle{abbrv}
\bibliography{CLT2-bib}

\begin{thebibliography}{10}

\bibitem{Amoroso}
F.~Amoroso and M.~Mignotte.
\newblock On the distribution of the roots of polynomials.
\newblock {\em Ann. Inst. Fourier (Grenoble)}, 46(5):1275--1291, 1996.

\bibitem{sample-Strong-R}
N.~Anari, S.~O. Gharan, and A.~Rezaei.
\newblock Monte {C}arlo {M}arkov chain algorithms for sampling strongly
  {R}ayleigh distributions and determinantal point processes.
\newblock In V.~Feldman, A.~Rakhlin, and O.~Shamir, editors, {\em 29th Annual
  Conference on Learning Theory}, volume~49 of {\em Proceedings of Machine
  Learning Research}, pages 103--115, Columbia University, New York, New York,
  USA, 23--26 Jun 2016. PMLR.

\bibitem{AOSS18}
N.~Anari, S.~Oveis~Gharan, A.~Saberi, and N.~Srivastava.
\newblock Approximating the largest root and applications to interlacing
  families.
\newblock In {\em Proceedings of the {T}wenty-{N}inth {A}nnual {ACM}-{SIAM}
  {S}ymposium on {D}iscrete {A}lgorithms}, pages 1015--1028. SIAM,
  Philadelphia, PA, 2018.

\bibitem{harmonic-axler}
S.~Axler, P.~Bourdon, and W.~Ramey.
\newblock {\em Harmonic function theory}, volume 137 of {\em Graduate Texts in
  Mathematics}.
\newblock Springer-Verlag, New York, second edition, 2001.

\bibitem{bender73}
E.~A. Bender.
\newblock Central and local limit theorems applied to asymptotic enumeration.
\newblock {\em J. Combinatorial Theory Ser. A}, 15:91--111, 1973.

\bibitem{bender83}
E.~A. Bender and L.~B. Richmond.
\newblock Central and local limit theorems applied to asymptotic enumeration.
  {II}. {M}ultivariate generating functions.
\newblock {\em J. Combin. Theory Ser. A}, 34(3):255--265, 1983.

\bibitem{BE}
W.~Bergweiler and A.~Eremenko.
\newblock Distribution of zeros of polynomials with positive coefficients.
\newblock {\em Ann. Acad. Sci. Fenn. Math.}, 40(1):375--383, 2015.

\bibitem{BES}
W.~Bergweiler, A.~Eremenko, and A.~Sokal.
\newblock Roots of polynomials with positive coefficients, 2013.

\bibitem{Bilu}
Y.~Bilu.
\newblock Limit distribution of small points on algebraic tori.
\newblock {\em Duke Math. J.}, 89(3):465--476, 1997.

\bibitem{BlochPolya}
A.~Bloch and G.~P\'{o}lya.
\newblock On the {R}oots of {C}ertain {A}lgebraic {E}quations.
\newblock {\em Proc. London Math. Soc. (2)}, 33(2):102--114, 1931.

\bibitem{BorBranI}
J.~Borcea and P.~Br\"{a}nd\'{e}n.
\newblock The {L}ee-{Y}ang and {P}\'{o}lya-{S}chur programs. {I}. {L}inear
  operators preserving stability.
\newblock {\em Invent. Math.}, 177(3):541--569, 2009.

\bibitem{BorBranII}
J.~Borcea and P.~Br\"{a}nd\'{e}n.
\newblock The {L}ee-{Y}ang and {P}\'{o}lya-{S}chur programs. {II}. {T}heory of
  stable polynomials and applications.
\newblock {\em Comm. Pure Appl. Math.}, 62(12):1595--1631, 2009.

\bibitem{BBL}
J.~Borcea, P.~Br\"{a}nd\'{e}n, and T.~M. Liggett.
\newblock Negative dependence and the geometry of polynomials.
\newblock {\em J. Amer. Math. Soc.}, 22(2):521--567, 2009.

\bibitem{Branden-Jonasson}
P.~Br\"{a}nd\'{e}n and J.~Jonasson.
\newblock Negative dependence in sampling.
\newblock {\em Scand. J. Stat.}, 39(4):830--838, 2012.

\bibitem{cauchy1828exercises}
A.~L.~B. Cauchy.
\newblock {\em Exercises de math{\'e}matiques}, volume~3.
\newblock Bure fr{\`e}res, 1828.
\newblock p. 122.

\bibitem{matroid-halfplane}
Y.-B. Choe, J.~G. Oxley, A.~D. Sokal, and D.~G. Wagner.
\newblock Homogeneous multivariate polynomials with the half-plane property.
\newblock {\em Advances in Applied Mathematics}, 32(1-2):88--187, 2004.

\bibitem{conway}
J.~B. Conway.
\newblock {\em Functions of one complex variable}, volume~11 of {\em Graduate
  Texts in Mathematics}.
\newblock Springer-Verlag, New York-Berlin, second edition, 1978.

\bibitem{Cramer-Wold}
H.~Cram\'{e}r and H.~Wold.
\newblock Some theorems on distribution functions.
\newblock {\em J. London Math. Soc.}, 11(4):290--294, 1936.

\bibitem{Cuesta-Albertos}
J.~A. Cuesta-Albertos, R.~Fraiman, and T.~Ransford.
\newblock A sharp form of the {C}ram\'{e}r-{W}old theorem.
\newblock {\em J. Theoret. Probab.}, 20(2):201--209, 2007.

\bibitem{angelis}
V.~De~Angelis.
\newblock Asymptotic expansions and positivity of coefficients for large powers
  of analytic functions.
\newblock {\em Int. J. Math. Math. Sci.}, (16):1003--1025, 2003.

\bibitem{dobrushin96}
R.~Dobrushin.
\newblock Estimates of semi-invariants for the ising model at low temperatures.
\newblock {\em Translations of the American Mathematical Society-Series 2},
  177:59--82, 1996.

\bibitem{erdos-polys}
P.~Erd\H{o}s.
\newblock On the uniform distribution of the roots of certain polynomials.
\newblock {\em Ann. of Math. (2)}, 43:59--64, 1942.

\bibitem{erdos-turan}
P.~Erd\H{o}s and P.~Tur\'{a}n.
\newblock On the distribution of roots of polynomials.
\newblock {\em Ann. of Math. (2)}, 51:105--119, 1950.

\bibitem{feller}
W.~Feller.
\newblock {\em An introduction to probability theory and its applications.
  {V}ol. {II}}.
\newblock Second edition. John Wiley \& Sons, Inc., New York-London-Sydney,
  1971.

\bibitem{chro-poly}
R.~Fern\'{a}ndez and A.~Procacci.
\newblock Regions without complex zeros for chromatic polynomials on graphs
  with bounded degree.
\newblock {\em Combinatorics, Probability and Computing}, 17(2):225–238,
  2008.

\bibitem{gao92}
Z.~Gao and L.~B. Richmond.
\newblock Central and local limit theorems applied to asymptotic enumeration.
  {IV}. {M}ultivariate generating functions.
\newblock volume~41, pages 177--186. 1992.
\newblock Asymptotic methods in analysis and combinatorics.

\bibitem{GLP}
S.~Ghosh, T.~M. Liggett, and R.~Pemantle.
\newblock Multivariate {CLT} follows from strong {R}ayleigh property.
\newblock In {\em 2017 {P}roceedings of the {F}ourteenth {W}orkshop on
  {A}nalytic {A}lgorithmics and {C}ombinatorics ({ANALCO})}, pages 139--147.
  SIAM, Philadelphia, PA, 2017.

\bibitem{godsil81}
C.~D. Godsil.
\newblock Matching behaviour is asymptotically normal.
\newblock {\em Combinatorica}, 1(4):369--376, 1981.

\bibitem{Granville}
A.~Granville.
\newblock The distribution of roots of a polynomial.
\newblock In {\em Equidistribution in number theory, an introduction}, volume
  237 of {\em NATO Sci. Ser. II Math. Phys. Chem.}, pages 93--102. Springer,
  Dordrecht, 2007.

\bibitem{Gurvits}
L.~Gurvits.
\newblock The van der {W}aerden conjecture for mixed discriminants.
\newblock {\em Adv. Math.}, 200(2):435--454, 2006.

\bibitem{Gurvits08}
L.~Gurvits.
\newblock Van der {W}aerden/{S}chrijver-valiant like conjectures and stable
  (aka hyperbolic) homogeneous polynomials: one theorem for all.
\newblock {\em The Electronic Journal of Combinatorics}, 15(1):66, 2008.

\bibitem{harper67}
L.~H. Harper.
\newblock Stirling behavior is asymptotically normal.
\newblock {\em Ann. Math. Statist.}, 38:410--414, 1967.

\bibitem{heilmann-lieb72}
O.~J. Heilmann and E.~H. Lieb.
\newblock Theory of monomer-dimer systems.
\newblock {\em Comm. Math. Phys.}, 25:190--232, 1972.

\bibitem{hurwitz64}
A.~Hurwitz.
\newblock On the conditions under which an equation has only roots with
  negative real parts.
\newblock {\em Selected papers on mathematical trends in control theory},
  65:273--284, 1964.

\bibitem{ising1979}
D.~Iagolnitzer and B.~Souillard.
\newblock Lee-{Y}ang theory and normal fluctuations.
\newblock {\em Phys. Rev. B (3)}, 19(3):1515--1518, 1979.

\bibitem{jackson-procacci-sokal}
B.~Jackson, A.~Procacci, and A.~D. Sokal.
\newblock Complex zero-free regions at large {$|q|$} for multivariate tutte
  polynomials (alias potts-model partition functions) with general complex edge
  weights.
\newblock {\em Journal of Combinatorial Theory, Series B}, 103(1):21 -- 45,
  2013.

\bibitem{Kahn00}
J.~Kahn.
\newblock A normal law for matchings.
\newblock {\em Combinatorica}, 20(3):339--391, 2000.

\bibitem{kyng2019}
R.~Kyng, K.~Luh, and Z.~Song.
\newblock Four deviations suffice for rank 1 matrices.
\newblock {\em arXiv preprint arXiv:1901.06731}, 2019.

\bibitem{kyng2018}
R.~Kyng and Z.~Song.
\newblock A matrix {C}hernoff bound for strongly {R}ayleigh distributions and
  spectral sparsifiers from a few random spanning trees.
\newblock In {\em 59th {A}nnual {IEEE} {S}ymposium on {F}oundations of
  {C}omputer {S}cience---{FOCS} 2018}, pages 373--384. IEEE Computer Soc., Los
  Alamitos, CA, 2018.

\bibitem{LPRS}
J.~L. Lebowitz, B.~Pittel, D.~Ruelle, and E.~R. Speer.
\newblock Central limit theorems, {L}ee-{Y}ang zeros, and graph-counting
  polynomials.
\newblock {\em J. Combin. Theory Ser. A}, 141:147--183, 2016.

\bibitem{lee-yang}
T.~D. Lee and C.~N. Yang.
\newblock Statistical theory of equations of state and phase transitions. {II}.
  {L}attice gas and {I}sing model.
\newblock {\em Phys. Rev. (2)}, 87:410--419, 1952.

\bibitem{li2017polynomial}
C.~Li, S.~Jegelka, and S.~Sra.
\newblock Polynomial time algorithms for dual volume sampling.
\newblock In {\em Advances in Neural Information Processing Systems}, pages
  5038--5047, 2017.

\bibitem{liggettstability}
T.~M. Liggett and A.~Vandenberg-Rodes.
\newblock Stability on {$\{0,1,2,\dots\}^S$}: birth-death chains and particle
  systems.
\newblock In {\em Notions of positivity and the geometry of polynomials},
  Trends Math., pages 311--329. Birkh\"{a}user/Springer Basel AG, Basel, 2011.

\bibitem{LittlewoodOfford}
J.~E. Littlewood and A.~C. Offord.
\newblock On the number of real roots of a random algebraic equation.
\newblock {\em J. London Math. Soc.}, 13(4):288--295, 1938.

\bibitem{LittlewoodOfford-II}
J.~E. Littlewood and A.~C. Offord.
\newblock On the number of real roots of a random algebraic equation. ii.
\newblock In {\em Mathematical Proceedings of the Cambridge Philosophical
  Society}, volume~35, pages 133--148. Cambridge University Press, 1939.

\bibitem{LittlewoodOfford-III}
J.~E. Littlewood and A.~C. Offord.
\newblock On the number of real roots of a random algebraic equation. {III}.
\newblock {\em Rec. Math. [Mat. Sbornik] N.S.}, 12(54):277--286, 1943.

\bibitem{Marcin}
J.~Marcinkiewicz.
\newblock Sur une propri{\'e}t{\'e} de la loi de {G}auss.
\newblock {\em Mathematische Zeitschrift}, 44(1):612--618, 1939.

\bibitem{InterlacingI}
A.~W. Marcus, D.~A. Spielman, and N.~Srivastava.
\newblock Interlacing families {I}: {B}ipartite {R}amanujan graphs of all
  degrees.
\newblock {\em Ann. of Math. (2)}, 182(1):307--325, 2015.

\bibitem{InterlacingII}
A.~W. Marcus, D.~A. Spielman, and N.~Srivastava.
\newblock Interlacing families {II}: {M}ixed characteristic polynomials and the
  {K}adison-{S}inger problem.
\newblock {\em Ann. of Math. (2)}, 182(1):327--350, 2015.

\bibitem{clt1}
M.~Michelen and J.~Sahasrabudhe.
\newblock Central limit theorems from the roots of probability generating
  functions.
\newblock {\em To appear, Adv. Math.}, 2018.

\bibitem{var-paper}
M.~Michelen and J.~Sahasrabudhe.
\newblock Large variance from zero free regions.
\newblock 2019.
\newblock In preperation.

\bibitem{Mignotte}
M.~Mignotte.
\newblock Remarque sur une question relative \`a des fonctions conjugu\'{e}es.
\newblock {\em C. R. Acad. Sci. Paris S\'{e}r. I Math.}, 315(8):907--911, 1992.

\bibitem{peres}
P.~M\"{o}rters and Y.~Peres.
\newblock {\em Brownian motion}, volume~30 of {\em Cambridge Series in
  Statistical and Probabilistic Mathematics}.
\newblock Cambridge University Press, Cambridge, 2010.
\newblock With an appendix by Oded Schramm and Wendelin Werner.

\bibitem{OG-TSP}
S.~Oveis~Gharan, A.~Saberi, and M.~Singh.
\newblock A randomized rounding approach to the traveling salesman problem.
\newblock In {\em 2011 {IEEE} 52nd {A}nnual {S}ymposium on {F}oundations of
  {C}omputer {S}cience---{FOCS} 2011}, pages 550--559. IEEE Computer Soc., Los
  Alamitos, CA, 2011.

\bibitem{pemantle2000}
R.~Pemantle.
\newblock Towards a theory of negative dependence.
\newblock volume~41, pages 1371--1390. 2000.
\newblock Probabilistic techniques in equilibrium and nonequilibrium
  statistical physics.

\bibitem{pemantle-survey}
R.~Pemantle.
\newblock Hyperbolicity and stable polynomials in combinatorics and
  probability.
\newblock In {\em Current developments in mathematics, 2011}, pages 57--123.
  Int. Press, Somerville, MA, 2012.

\bibitem{pemantle}
R.~Pemantle.
\newblock Personal Communication, 2017.

\bibitem{pemantle-peres}
R.~Pemantle and Y.~Peres.
\newblock Concentration of {L}ipschitz functionals of determinantal and other
  strong {R}ayleigh measures.
\newblock {\em Combin. Probab. Comput.}, 23(1):140--160, 2014.

\bibitem{pemantle04}
R.~Pemantle and M.~C. Wilson.
\newblock Asymptotics of multivariate sequences. {II}. {M}ultiple points of the
  singular variety.
\newblock {\em Combin. Probab. Comput.}, 13(4-5):735--761, 2004.

\bibitem{pemantle08}
R.~Pemantle and M.~C. Wilson.
\newblock Twenty combinatorial examples of asymptotics derived from
  multivariate generating functions.
\newblock {\em SIAM Rev.}, 50(2):199--272, 2008.

\bibitem{peters-regts17}
H.~Peters and G.~Regts.
\newblock On a conjecture of {S}okal concerning roots of the independence
  polynomial.
\newblock {\em The Michigan Mathematical Journal}, 68, 01 2017.

\bibitem{regts18}
G.~Regts.
\newblock Zero-free regions of partition functions with applications to
  algorithms and graph limits.
\newblock {\em Combinatorica}, 38(4):987--1015, 2018.

\bibitem{rota-interview}
G.-C. Rota and D.~Sharp.
\newblock Mathematics, philosophy, and artificial intelligence: a dialogue with
  {G}ian-{C}arlo {R}ota and {D}avid {S}harp.
\newblock {\em Los Alamos Science}, Spring/Summer 1985.

\bibitem{rucinski84}
A.~Ruci\'{n}ski.
\newblock The behaviour of {$\binom{n}{k,\cdots,k,n-ik}c^i/i!$} is
  asymptotically normal.
\newblock {\em Discrete Math.}, 49(3):287--290, 1984.

\bibitem{Schur}
I.~Schur.
\newblock Untersuchungen \"{u}ber algebraische gleichungen.
\newblock {\em Sitz. Preuss. Akad. Wiss., Phys.- Math. Kl.}, pages 403--428,
  1933.

\bibitem{scott-sokal05}
A.~D. Scott and A.~D. Sokal.
\newblock The repulsive lattice gas, the independent-set polynomial, and the
  {L}ov{\'a}sz local lemma.
\newblock {\em Journal of Statistical Physics}, 118(5-6):1151--1261, 2005.

\bibitem{scott-sokal06}
A.~D. Scott and A.~D. Sokal.
\newblock On dependency graphs and the lattice gas.
\newblock {\em Combinatorics, Probability and Computing}, 15(1-2):253--279,
  2006.

\bibitem{shearer85}
J.~B. Shearer.
\newblock On a problem of spencer.
\newblock {\em Combinatorica}, 5(3):241--245, 1985.

\bibitem{Sound}
K.~Soundararajan.
\newblock Equidistribution of zeros of polynomials.
\newblock {\em Amer. Math. Monthly}, 126(3):226--236, 2019.

\bibitem{Szego}
G.~Szeg{\"o}.
\newblock {\em Bemerkungen zu einem Satz von E. Schmidt {\"u}ber algebraische
  Gleichungen}.
\newblock Verlag der Akademie der Wissenschaften in Kommission bei Walter de
  Gruyter u~…, 1934.

\bibitem{wagner05}
D.~G. Wagner.
\newblock Matroid inequalities from electrical network theory.
\newblock {\em The Electronic Journal of Combinatorics}, 11(2):1, 2005.

\bibitem{wagner09}
D.~G. Wagner.
\newblock Weighted enumeration of spanning subgraphs with degree constraints.
\newblock {\em Journal of Combinatorial Theory, Series B}, 99(2):347 -- 357,
  2009.

\bibitem{yang-lee}
C.~N. Yang and T.~D. Lee.
\newblock Statistical theory of equations of state and phase transitions. {I}.
  {T}heory of condensation.
\newblock {\em Phys. Rev. (2)}, 87:404--409, 1952.

\bibitem{random-Geo}
D.~Yogeshwaran, E.~Subag, and R.~J. Adler.
\newblock Random geometric complexes in the thermodynamic regime.
\newblock {\em Probab. Theory Related Fields}, 167(1-2):107--142, 2017.

\end{thebibliography}

\appendix

\section{A few displaced proofs} \label{app:proofs}

\begin{proof}[Proof of Lemma~\ref{lem:harnack-general}]
	This follows from iterating \eqref{eq:harnack-ball}: we can find $k = \lfloor d / (\eps / 2) \rfloor$ points $y_1,\ldots, y_k$ on the line segment from $z_1$ to $z_2$ so that $d(z_1,y_1) \leq \eps/2$, $d(y_j,y_{j+1}) \leq \eps/2$ and $d(y_k,z_2) \leq \eps/2$.  Applying \eqref{eq:harnack-ball} $k + 1$ times completes the proof.
\end{proof}

\begin{proof}[Proof of Lemma \ref{lem:ratioOfPoisson}]
	Since rescaling $\eps$ only rescales the corresponding Poisson densities, we may assume without loss of generality that $\eps = 1$; similarly, we may recenter both balls to the origin.  Fix $w \in \partial B(0,2)$ and $z \in B(0,1)$.  By \cite[eq. 1.15]{harmonic-axler} we have that $$P_z(w) = \frac{1}{4\pi}\frac{4 - |z|^2}{|z - w|^2}\,.$$
	
	Thus $$\frac{P_z(w)}{P_1(w)} = \frac{4-|z|^2}{|z - w|^2} \leq \frac{4 - |z|^2}{(2 - |z|)^2} = \frac{2+|z|}{2 - |z|} \leq 3\,.$$ 
	
\end{proof}

\begin{proof}[Proof of Fact~\ref{fact:negcos}]
	For (\ref{equ:Factdiff-ve}), we consider $\t = (2\pi k)/j$, where $k$ is chosen so that so that
	$| (2\pi k)/j - \pi/2 | \leq \pi /j $. Hence
	\[  ( \cos \t )^j - \cos j\t < (\pi/j)^{j} - 1 ,\] 
	which is less than $-1/2$ for $j \geq 4$. For $j=3$ choose $\t = 4 \pi /3 $ so that $(\cos \t)^3 <0$ and $\cos(3 \t) = 1$. 
	To see (\ref{equ:Factdiff+ve}), we evaluate at $\t = \pi/j$ to see
	\[  ( \cos \t )^j - \cos j\t >  1, \] provided $j \geq 3$. 
\end{proof}

\vspace{4mm}

The following is a theorem of Cuesta-Albertos, Fraiman and Ransford from \cite{Cuesta-Albertos}, which we shall use to derive Corollary~\ref{cor:CAFR}.

\begin{theorem}[\hspace{1sp}\cite{Cuesta-Albertos}] \label{thm:CAFR} Let $\nu$ be a Borel probability measure on $\R^d$ for which the absolute moments 
\[ M_k := \int_{\R^d} \|x\|_2^k \, d\nu,  
\] are finite and satisfy the Carleman condition $\sum_{k\geq 1} M_k^{-1/k} = \infty$. If $\tilde{\nu}$ is a Borel probability measure on $\R^d$ for which
$\Pi(\nu,\tilde{\nu})$ is not contained in the zero-set of a non-trivial polynomial then $\nu=\tilde{\nu}$.
\end{theorem}

\begin{proof}[Proof of Corollary~\ref{cor:CAFR}]
	For each $d\times d$ positive-semi definite matrix $A$, the moment generating function $\int_{\R^d} e^{t\|x\|_2} d\nu_A$ is finite for all $t$.  This implies that that the Carleman condition $\sum_{k \geq 1} M_k^{-1/k} = \infty$ is satisfied; indeed,
	\[  \int_{\R^d} e^{t\|x\|_2} d\nu_A = \sum_{k\geq 0} \frac{t^k}{k!}M_k ,
	\] and so $\limsup_{k} (M_k /k!)^{1/k} < C <\infty$. Thus for large enough $k$ we have 
	\[ M_k^{-1/k} > 1/(2Ck!^{1/k}) > 1/(2Ck),
	\] which implies the Carleman condition.
	
	Finally note that $\Z_{\geq 0}^d$ is not contained in the zero-set of a non-trivial polynomial, showing that the hypotheses of Theorem \ref{thm:CAFR} are met. 
\end{proof}

\begin{proof}[Proof of Fact \ref{fact:tight}]  Tightness follows from Chebyshev's inequality:
	
	$$\PP( \|X^\ast_{n} \|_2 \geq x ) = \PP( \|X_{n} - \mu_n\|_2 \geq x \sigma_n) \leq \frac{\E \|X_n - \mu_n\|_2^2}{(x \s_n^2)} = \frac{\mathrm{Tr}(A_n)}{x^2 \s_n^2} \leq \frac{d}{x^2}$$
	due to the fact that $\s_n^2$ is the operator norm of $A_n$.
\end{proof}

\begin{proof}[Proof of Lemma \ref{lem:analytic}]
	To prove analyticity of $f_X(z)$, we use Morera's theorem, i.e.~show that $\int_\gamma f_X(z)\,dz = 0$ for each closed $C^1$ curve $\gamma$ in $\C \setminus \R_{\leq 0}$.  Provided we can swap integrals, we compute	
	$$ \int_\gamma f_X(z) \,dz = \int_\gamma \E [z^X] \,dz = \E\int_{\gamma} z^X \, dz  = 0 $$
	where the last equality is because $z^x$ is analytic in $\C\setminus \R_{\leq 0}$ for each value of $x$.  To see that we can swap integrals, find values $0 < r < R < \infty$ so that $|z| \in [r,R]$ for all $z \in \gamma$.  Then $$ \int_\gamma |\E[z^X]| \,dz \leq \int_\gamma \E[ r^X + R^X] \,dz \leq |\gamma| \cdot (\E[ r^X] + \E[R^X]) < \infty $$
	where $|\gamma|$ is the length of $\gamma$.  Fubini's theorem then justifies the swapping of integrals.
\end{proof}

\section{Planar geometry calculations}\label{app:kernel}

Here we prove a handful of straightforward planar geometry calculations that have allowed us to nest truncated sectors inside balls and vice-versa:

\begin{lemma} \label{lem:planar-geometry1}
	Let $\eps \leq \frac{1}{2}$.  Then each point $z \in B(1,\eps)$ satisfies $|z| \in [1-\eps,1+\eps]$ and $|\arg(z)| \leq 2\eps$.  
\end{lemma}
\begin{proof}
	We note that $|z| \leq 1 + |1 - z| \leq 1 + \eps$ and $|z| \geq 1 - |1 - z| \geq 1 - \eps$, thereby showing the modulus bound.  For the argument bound, it is sufficient to show that the line $\{\arg(\zeta) = 2\eps\}$ does not intersect $B(1,\eps)$.  We note that the distance squared from a point $t e^{i2\eps}$ to $1$ is $$(t \cos(2\eps) - 1)^2 + (t\sin(2\eps))^2 = t^2 - 2t \cos(2\eps) +1\,. $$
	
	This quadratic achieves its minimum at $t = \cos(2\eps)$, implying that the minimum distance is $\sin(2\eps)$.  Using the inequality $\sin(\theta) \geq \theta - \theta^3 / 3$ for $\theta \geq 0$ together with the inequality $2\eps - 8 \eps^3 / 3 \geq \eps$ for $\eps \leq \frac{1}{2}$ shows that this distance is at least $ \eps$ for $\eps \leq \frac{1}{2}$.  
\end{proof}

\begin{corollary} \label{lem:planar-geometry3}
	Let $\eps \leq 1$.  Then each point $z \in B(1,\eps/2)$ satisfies $|z| \in [\frac{1}{1+\eps},1+\eps]$ and $|\arg(z)| \leq \eps$.  
\end{corollary}
\begin{proof}
	The bounds on argument as well as the upper bound on modulus follow from Lemma \ref{lem:planar-geometry1}; for the lower bound on modulus, note that the modulus is minimized for $z = 1 - \eps/2$, which is indeed larger than $\frac{1}{1 + \eps}$ provided $\eps \leq 1$.
\end{proof}

\begin{lemma} \label{lem:planar-geometry2}
	Fix $\eps \leq 1$.  Then for each point $z$ with $|z| \in [1-\eps,1+\eps]$ and $|\arg(z)| \leq \eps$, we have $|z - 1| \leq 2\eps$.
\end{lemma}
\begin{proof}
	Note first that the point in the region $\{z : |z| \in [1 - \eps,1 + \eps], |\arg(z)| \leq \eps  \}$ of maximum distance to $1$ must be one of $(1+\eps) e^{i\eps}$, $(1 - \eps)e^{i\eps}$ or one of their conjugates.  Computing the distance square of the former to $1$, we see \begin{align*}
	|(1+\eps) e^{i\eps} - 1|^2 =  \eps^2 + 2 \eps(1 - \cos(\eps)) + 2(1 - \cos(\eps)) \leq \eps^2 (2 + \eps) \leq 4 \eps^2\,.
	\end{align*}
	
	Similarly, $$|(1 - \eps) e^{i \eps} - 1|^2 = \eps^2 - 2\eps(1- \cos(\eps)) + 2(1 - \cos(\eps)) \leq 2 \eps^2 \leq 4 \eps^2\,.$$
\end{proof}

\begin{lemma} \label{lem:planar-geometry4}
	For each $\eps < 1/2$,  $\{ e^w : w \in B(0,\eps) \} \subset B(1,2\eps)$.
\end{lemma}
\begin{proof}
	For $w \in B(0,\eps)$, write $w = a + ib$.  Then $$|e^w - 1|^2 = e^{2a} - 2e^a \cos(b) + 1 \leq e^{2a} - 2e^a(1 - b^2/2) + 1 = (e^a - 1)^2 + e^a b^2\,.$$
	
	Using the inequalities $(e^a - 1)^2 \leq 4 a^2$ and $e^a < 4$ for $a < 1$ yields $$|e^w - 1|^2 \leq 4 a^2 + 4 b^2 = 4 \eps^2\,.$$ 
\end{proof}

\end{document}